\documentclass{amsart}
 \usepackage{amsmath,amssymb,amsfonts,amsthm,amscd,textcomp,times,booktabs,mathrsfs}
   \usepackage[dvips]{graphicx}
    \usepackage{braket}
  \usepackage{color,float}
  \usepackage{bm}
\usepackage[all]{xy}
\newtheorem{theorem}{Theorem}

\newtheorem{proposition}[theorem]{Proposition}
\newtheorem{lemma}[theorem]{Lemma}
\newtheorem{definition}[theorem]{Definition}
\newtheorem{corollary}[theorem]{Corollary}
\newtheorem{remark}[theorem]{Remark}

\numberwithin{equation}{section}
\numberwithin{theorem}{section}
\numberwithin{table}{section}

\newcommand{\Z}{\ensuremath{\mathbb{Z}}}

\newcommand{\R}{\ensuremath{\mathbb{R}}}
\newcommand{\C}{\ensuremath{\mathbb{C}}}
\newcommand{\I}{\ensuremath{\sqrt{-1}}}

\newcommand{\GL}{\ensuremath{\mathrm{GL}}}

\newcommand{\SO}{\ensuremath{\mathrm{SO}}}
\newcommand{\Spin}{\ensuremath{\mathrm{Spin(7)}}}
\newcommand{\SU}{\ensuremath{\mathrm{SU}}}

\newcommand{\Symp}{\ensuremath{\mathrm{Sp}}}

\newcommand{\del}{\ensuremath{\mathrm{\partial}}}

\newcommand{\der}{\ensuremath{\mathrm{d}}}

\newcommand{\Derbar}{\ensuremath{\mathrm{\overline{D}}}}
\newcommand{\id}{\ensuremath{\mathrm{id}}}

\newcommand{\norm}[1]{\ensuremath{\left| #1 \right|}}
\newcommand{\Norm}[1]{\ensuremath{\left\| #1 \right\|}}
\newcommand{\restrict}[2]{\ensuremath{\left. #1 \right|_{#2}}}

\DeclareMathOperator{\Real}{Re}

\DeclareMathOperator*{\Hol}{Hol}

\DeclareMathOperator{\Sing}{Sing}
\DeclareMathOperator{\Proj}{Proj}
\DeclareMathOperator{\Bl}{Bl}

\def\a{\alpha}

\def\p{\partial}
\def\z{\zeta}

\def\Ahat{\widehat{A}}
\def\P{\mathbb{P}}

\begin{document}
\title{Gluing construction of compact $\Spin$-manifolds}

\author{Mamoru Doi}

 \email{doi.mamoru@gmail.com}

\author{Naoto Yotsutani}

\address{School of Mathematical Sciences at Fudan University,
Shanghai, 200433, P. R. China}
\email{naoto-yotsutani@fudan.edu.cn}

 \dedicatory{}
\date{\today}
\maketitle
\noindent{\bfseries Abstract.}
We give a differential-geometric construction of compact manifolds with holonomy
$\mathrm{Spin}(7)$  which is based on Joyce's second construction of compact $\mathrm{Spin}(7)$-manifolds in \cite{Joyce00}
and Kovalev's gluing construction of $G_2$-manifolds in \cite{Kovalev03}.
We also give some examples of compact $\mathrm{Spin}(7)$-manifolds, at least one of which is \emph{new}.
Ingredients in our construction are \emph{orbifold admissible pairs with} a compatible antiholomorphic involution. 
Here in this paper we need orbifold admissible pairs $(\overline{X}, D)$ consisting of a 
four-dimensional compact K\"{a}hler orbifold $\overline{X}$
with isolated singular points modelled on $\mathbb{C}^4/\mathbb{Z}_4$, and a smooth anticanonical divisor $D$ on $\overline{X}$.
Also, we need a compatible antiholomorphic involution $\sigma$ on $\overline{X}$ which fixes the singular points in $\overline{X}$ and 
acts freely on the anticanoncial divisor $D$. If two orbifold admissible pairs $(\overline{X}_1, D_1)$, $(\overline{X}_2, D_2)$ with 
$\dim_{\mathbb{C}} \overline{X}_i = 4$ and compatible antiholomorphic involutions $\sigma_i$ on $\overline{X}_i$ satisfy the 
\emph{gluing condition}, we can glue $(\overline{X}_1 \setminus D_1)/\braket{\sigma_1}$ and 
$(\overline{X}_2 \setminus D_2)/\braket{\sigma_2}$ together to obtain a compact Riemannian $8$-manifold $(M, g)$ 
whose holonomy group $\mathrm{Hol}(g)$ is contained in $\mathrm{Spin}(7)$. Furthermore, if the $\widehat{A}$-genus of $M$ equals $1$,
then $M$ is a $\mathrm{Spin}(7)$-manifold, i.e., a compact Riemannian manifold with holonomy $\mathrm{Spin}(7)$. 
We shall investigate our gluing construction using orbifold admissible pairs $(\overline{X}_i,D_i)$ with $i=1,2$ when $D_1=D_2=D$ and $D$ is a complete intersection in 
a weighted projective space, as well as when $(\overline{X}_1,D_1)=(\overline{X}_2,D_2)$ and $\sigma_1=\sigma_2$ (the \emph{doubling} case).

\section{Introduction}
The purpose of this paper is to give a gluing construction and examples of compact $\Spin$-manifolds, i.e., compact Riemannian manifolds with
holonomy $\Spin$. 
We constructed in our previous papers \cite{DY14} and \cite{DY15}
Calabi-Yau threefolds and fourfolds by gluing together two asymptotically cylindrical Ricci-flat K\"{a}hler manifolds,
using the gluing technique which Kovalev used in constructing compact $G_2$-manifolds \cite{Kovalev03}.
Such asymptotically cylindrical Ricci-flat K\"{a}hler manifolds $X$ are obtained from admissible pairs $(\overline{X},D)$
by setting $X=\overline{X}\setminus D$.
In the present paper we glue instead two asymptotically cylindrical $\Spin$-\emph{orbifolds} to construct a compact $\Spin$-orbifold,
and then resolve the singularities to obtain a compact $\Spin$-manifold.
Such asymptotically cylindrical $\Spin$-orbifolds are obtained 
by setting $(\overline{X}\setminus D)/\braket{\sigma}$
from \emph{orbifold admissible pairs} $(\overline{X},D)$ 
with isolated singular points modelled on $\C^4/\Z_4$, and a compatible antiholomorphic involution $\sigma$ on $\overline{X}$.
Originally, Joyce resolved $X=\overline{X}\setminus D$ to obtain compact $\Spin$-manifolds 
when $\overline{X}$ is a four-dimensional Calabi-Yau orbifold and $D=\emptyset$, so that $X=\overline{X}$ is \emph{compact}:
Beginning with a compact four-dimensional Calabi-Yau
orbifold $\overline{X}$ with isolated singular points modelled on $\C^4/\Z_4$, 
and an antiholomorphic involution $\sigma$ on $\overline{X}$
with $(\overline{X})^\sigma =\Sing\overline{X}$, Joyce proved first that $Z=\overline{X}/\braket{\sigma}$ admits a torsion-free $\Spin$-structure.
Since the associated Riemannian metric is flat (Euclidean) around the singularites of $Z$, he then replaced 
the neighborhood of each singularity of $Z$ with a suitable asymptotically locally Euclidean (ALE) $\Spin$-manifold to
obtain a family of simply-connected, smooth $8$-manifolds $\set{M^\epsilon}$ for $\epsilon\in(0,1]$ 
with a $\Spin$-structure $\Phi^\epsilon$ \emph{with small torsion},
which satisfies $\der\Phi^\epsilon\rightarrow 0$ as $\epsilon\rightarrow 0$ in a suitable sense.
Finally, he proved that $\Phi^\epsilon$ can be deformed to a torsion-free $\Spin$-structure
for sufficiently small $\epsilon$ using the analysis on $\Spin$-structures, 
so that $M=M^\epsilon$ admits a Riemannian metric with holonomy $\Spin$.
We note that asymptotically cylindrical $\Spin$-\emph{manifolds} are recently constructed by Kovalev in \cite{Kovalev13}
by resolving $(\overline{X}\setminus D)/\braket{\sigma}$.

In our construction, we begin with two orbifold admissible pairs $(\overline{X}_1,D_1)$
and $(\overline{X}_2,D_2)$, consisting of a compact K\"{a}hler orbifold $\overline{X}_i$ with $\dim_\C\overline{X}_i=4$ and 
a smooth anticanonical divisor $D_i$ on $\overline{X}_i$.
Also, we consider an antiholomorphic involution $\sigma_i$ acting on each $\overline{X}_i$.
As in Joyce's second construction of compact $\Spin$-manifolds, we require that $\overline{X}_i$ have isolated singular points
modelled on $\C^4/\Z_4$, and $(\overline{X}_i)^\sigma =\Sing\overline{X}_i$ 
(see Definintions $\ref{def:admissible}$ and $\ref{def:compatible}$).
In addition, we suppose that $\sigma$ preserves and acts freely on $D$.
Then by the existence result of
an asymptotically cylindrical Ricci-flat K\"{a}hler form on $\overline{X}_i\setminus D_i$,
each $\overline{X}_i\setminus D_i$ has a natural $\sigma_i$-invariant asymptotically cylindrical
torsion-free $\Spin$-structure, which pushes down to a torsion-free $\Spin$-structure $\Phi_i$ on $\overline{X}_i/\braket{\sigma_i}$. 
Now suppose the \emph{asymptotic models} $\left( (D_i\times S^1)/\braket{\sigma_{D_i\times S^1,\rm cyl}}\times\R_+,\Phi_{i,\rm{cyl}}\right)$
of $((\overline{X}_i\setminus D_i)/\braket{\sigma_i},\Phi_i)$ are isomorphic in a suitable sense,
which is ensured by the \emph{gluing condition} defined later (see Section $\ref{sec:gluing_cond}$).
Then as in Kovalev's construction in \cite{Kovalev03}, we can glue together
$(\overline{X}_1\setminus D_1)/\braket{\sigma_1}$ and $(\overline{X}_2\setminus D_2)/\braket{\sigma_2}$ 
along their cylindrical ends $(D_1\times S^1)/\braket{\sigma_{D_1\times S^1,\rm cyl}}\times (T-1,T+1)$ and
$(D_2\times S^1)/\braket{\sigma_{D_2\times S^1,\rm cyl}}\times (T-1,T+1)$, to obtain a compact $8$-\emph{orbifold} $M_T^\triangledown$.
Also, we can glue together the torsion-free $\Spin$-structures 
$\Phi_i$ on $(\overline{X}_i\setminus D_i)/\braket{\sigma_i}$
to construct a $\der$-closed $4$-form $\widetilde{\Phi}_T^\triangledown$ on $M_T^\triangledown$.
Furthermore, replacing each neighborhood of singular points in $M_T^\triangledown$ with a certain ALE $\Spin$-manifold, 
we construct a family $(M_T^\epsilon, \widetilde{\Phi}_T^\epsilon)$ of 
simply-connected, smooth $8$-manifolds with a $\der$-closed $4$-form for sufficiently small $\epsilon >0$,
such that each $\widetilde{\Phi}_T^{\epsilon}$ is projected to a $\Spin$-structure 
$\Phi_T^\epsilon =\Theta (\widetilde{\Phi}_T^\epsilon )$,
with $\Phi_T^\epsilon\to 0$ as $T\to\infty$ or $\epsilon\to 0$ in a suitable sense.
Now set $\epsilon =e^{-\gamma T}$ for some $\gamma >0$, and 
$(M^\epsilon,\Phi^\epsilon )=(M_T^\epsilon,\widetilde{\Phi}_T^\epsilon )$.
Then using the analysis on $\Spin$-structures by Joyce \cite{Joyce00},
we shall prove that $\Phi^\epsilon$ can be deformed into a torsion-free $\Spin$-structure
for sufficiently small $\epsilon$, so that
the resulting compact manifold $M^\epsilon$ admits a Riemannian metric with holonomy contained in $\Spin$.
Since $M=M^\epsilon$ is simply-connected, the $\Ahat$-genus $\Ahat(M)$ of $M$ is $1,2,3$ or $4$, 
and the holonomy group is determined as $\Spin , \SU(4), \Symp (2), \Symp(1)\times \Symp(1)$ 
respectively (see Theorem $\ref{thm:A-hat}$).
Hence if $\widehat{A}(M)=1$, then $M$ is a compact $\Spin$-manifold.

Beginning with a Fano four-orbifold $V$ with a smooth anticanonical divisor $D$, 
Kovalev obtained an orbifold admissible pair $(\overline{X},D)$ \emph{of Fano type} as follows.
Let $S$ be a smooth complex surface in $D$ representing the self-intersection class $D\cdot D$ on $V$.
If we take $\overline{X}$ to be the blow-up of $V$ along $S$,
then the proper transform of $D$ in $\overline{X}$, which is isomorphic to $D$ and denoted by $D$ again,
is an anticanonical divisor on $\overline{X}$ with the holomorphic normal bundle $N_{D/\overline{X}}$ trivial.
Hence $(\overline{X},D)$ is an orbifold admissible pair (Theorem \ref{thm:Fano}).

For a given orbifold admissible pair $(\overline{X}_1,D_1)$ with a compatible antiholomorphic involution $\sigma_1$,
it is difficult in general to find another admissible pair $(\overline{X}_2,D_2)$ with $\sigma_2$ such that
both $(\overline{X}_i\setminus D_i)/\braket{\sigma_i}$ have the same asymptotic model.
One way to solve this is the `doubling' method used in \cite{DY14}, \cite{DY15}, in which we take
$(\overline{X}_1,D_1)=(\overline{X}_2,D_2)$ and $\sigma_1=\sigma_2$.
For another solution, we investigate orbifold admissible pairs $(\overline{X},D)$ of Fano type 
when $V$ is a complete intersection in a weighted projective space $W=\C P^{k+3}(a_0,\dots,a_{k+3})$ with $k\geqslant 2$.
Suppose $\sigma$ is an antiholomorphic involution on $W$ and 
\begin{equation*}
V=\set{[\mathbf{z}]\in W|f_1(\mathbf{z})=\dots =f_{k-1}(\mathbf{z})=0},\qquad
D=\set{[\mathbf{z}]\in W|f_1(\mathbf{z})=\dots =f_k(\mathbf{z})=0},
\end{equation*}
where $D$ is smooth and $f_i$ are weighted homogeneous polynomials 
satisfying $\deg f_1+\dots +\deg f_k=a_0+\dots +a_{k+3}$ and $\sigma^*f_i=\overline{f_i}$ for $i=1,\dots ,k$.
Then by the adjunction formula, $V$ is a Fano four-orbifold with an anticanonical Calabi-Yau divisor $D$.
Choosing $f_{k+1}$ so that 
\begin{align*}
&\deg f_{k+1}=\deg f_k, \qquad\sigma^*f_{k+1}=\overline{f_{k+1}}\qquad\text{and}\\ 
&S=\set{[\mathbf{z}]\in W|f_1(\mathbf{z})=\dots =f_{k+1}(\mathbf{z})=0}\quad\text{represents}\quad D\cdot D,
\end{align*}
we have an orbifold admissible pair $(\overline{X}_1,D_1)$ with a compatible antiholomorphic involution $\sigma_1$
such that $(D_1,\restrict{\sigma_1}{D_1})$ is isomorphic to $(D,\restrict{\sigma}{D})$.
Meanwhile, if we exchange $f_{k}$ and $f_{k-1}$ (and choose suitable $f_{k+1}$ correspondingly),
then $V$ may change, but $D$ \emph{does not change}.
Hence we have another orbifold admissible pair $(\overline{X}_2,D_2)$ with $\sigma_2$ 
which has the same asymptotic model.

In the present paper, we shall give $3$ topologically distinct compact $\Spin$-manifolds, at least one of which is
\emph{new}. Each of the examples satisfies $b^2(M)=b^3(M)=0$ and $\widehat{A}(M)=1$.

In order to show $\widehat{A}(M)=1$, we shall make an approach similar to \cite{DY15}, Section $4.4$,
that is, we reduce the problem to the computations on the cohomology groups of $D$ and $S$. 
Using some results on weighted complete intersections
in \cite{Fletcher00}, we conclude that each resulting $8$-manifold $M$ satisfies $\widehat{A}(M)=1$.
Betti numbers $(b^2,b^3,b^4)$ of the compact $\Spin$-manifolds in our construction are
$(0,0,910), (0,0,1294)$ and $(0,0,1678)$. Of these compact $\Spin$-manifolds, the resulting manifold $M$
with $\chi(M)=1680$ is at least one new example which is not diffeomorphic to the known ones (see Theorem \ref{th:new spin(7)}).

This paper is organized as follows. 
Section $\ref{sec:Spin-structures}$ is a brief review of $\Spin$-structures. 
In Section $\ref{sec:gluing_procedure}$ we define orbifold admissible pairs which will be
ingredients in our gluing construction of compact $\Spin$-manifolds. This section is the heart of this paper. 
We consider compatible antiholomorphic 
involutions $\sigma$ on orbifold admissible pairs $(\overline{X},D)$ and glue together two orbifold admissible pairs with 
$\dim_{\C}\overline{X}=4$ divided by $\sigma$. The gluing theorems are stated in Section $\ref{sec:gluing_theorems}$ 
including both cases of $\Spin$-manifolds and Calabi-Yau fourfolds. 
Giving a quick review of basics on weighted projective spaces in Section $\ref{sec:prelim}$,
we obtain in Section $\ref{sec:CI_WPS}$ orbifold admissible pairs from complete intersections in weighted projective spaces.
Then in Section $\ref{sec:Spin(7)}$ we give a new example of compact $\Spin$-manifolds $M$ with
the Euler characteristic $\chi(M)=1680$. We also find Betti numbers $(b^2,b^3,b^4)$ of this $\Spin$-manifold 
by applying the Mayer-Vietoris theorem.
In the last section we shall give other examples of compact $\Spin$-manifolds taking weighted complete intersections in $\C P^5(1,1,1,1,4,4)$.
All the resulting compact $\Spin$-manifolds are listed in Table $\ref{T2}$. Finally we illustrate an example of the doubling construction of 
Calabi-Yau fourfolds (Corollary $\ref{cor:doubling_CY}$) from orbifold admissible pairs.

\noindent {\bfseries Acknowledgements.}
The first author is grateful to Professors Xiuxiong Chen and Xu Bin to give the opportunity to visit University of Science and Technology of China, Hefei
in April, $2011$ and discuss the authors' joint research projects with them.
The second author is also grateful to Professors Xiuxiong Chen, Xu Bin, Shengli Kang and Haozhao Li 
for their support and encouragement when he was in USTC.

\section{Geometry of $\Spin$-structures}\label{sec:Spin-structures}
Here we shall recall some basic facts about $\Spin$-structures on oriented $8$-manifolds.
The material in this section is also discussed in \cite{DY15}, Section $2$.
For more details, see \cite{Joyce00}, Chapter $10$.

We begin with the definition of $\Spin$-structures on oriented vector spaces of dimension $8$.
\begin{definition}\rm
Let $V$ be an oriented real vector space of dimension $8$.
Let $\{\bm{\theta}^1,\dots ,\bm{\theta}^8\}$ be an oriented basis of $V$.
Set
\begin{equation*}\label{eq:Phi0-g0}
\begin{aligned}
\bm{\Phi}_0=&\bm{\theta}^{1234}+\bm{\theta}^{1256}+\bm{\theta}^{1278}+\bm{\theta}^{1357}
-\bm{\theta}^{1368}-\bm{\theta}^{1458}-\bm{\theta}^{1467}\\
&-\bm{\theta}^{2358}-\bm{\theta}^{2367}-\bm{\theta}^{2457}+\bm{\theta}^{2468}
+\bm{\theta}^{3456}+\bm{\theta}^{3478}+\bm{\theta}^{5678},\\
\mathbf{g}_0=&\sum_{i=1}^8\bm{\theta}^i\otimes\bm{\theta}^i,
\end{aligned}
\end{equation*}
where $\bm{\theta}^{ij\dots k}=\bm{\theta}^i\wedge\bm{\theta}^j\wedge\dots\wedge\bm{\theta}^k$.
Define the $\GL_+(V)$-orbit spaces
\begin{align*}
\mathcal{A}(V)&=\Set{a^*\bm{\Phi}_0|a\in\GL_+(V)},\\
\mathcal{M}et(V)&=\Set{a^*\mathbf{g}_0|a\in\GL_+(V)}.
\end{align*}
\end{definition}
We call $\mathcal{A}(V)$ the set of \emph{Cayley $4$-forms}
(or the set of \emph{$\Spin$-structures}) on $V$.
On the other hand, $\mathcal{M}et(V)$ is the set of positive-definite inner products on $V$,
which is also a homogeneous space isomorphic to $\GL_+(V)/\SO(V)$, where $\SO(V)$ is defined by
\begin{equation*}
\SO(V)=\Set{a\in\GL_+(V)|a^*\mathbf{g}_0=\mathbf{g}_0}.
\end{equation*}

Now the group $\Spin$ is defined as the isotropy of the action of $\GL(V)$ (in place of $\GL_+(V)$)
on $\mathcal{A}(V)$ at $\bm{\Phi}_0$:
\begin{equation*}
\Spin =\Set{a\in\GL(V)|a^*\bm{\Phi}_0=\bm{\Phi}_0}.
\end{equation*}
Then one can show that $\Spin$ is a compact Lie group of dimension $27$ which is a
Lie subgroup of $\SO(V)$ (see \cite{Harvey90}).
Thus we have a natural projection
\begin{equation*}\label{eq:G2metric}
\xymatrix{\mathcal{A}(V)\cong\GL_+(V)/\Spin\ar@{>>}[r]&\GL_+(V)/\SO(V)\cong\mathcal{M}et(V)},
\end{equation*}
so that each Cayley $4$-form (or $\Spin$-structure) $\bm{\Phi}\in\mathcal{A}(V)$
defines a positive-definite inner product $\mathbf{g}_{\bm{\Phi}}\in\mathcal{M}et(V)$ on $V$.
\begin{definition}\label{def:T(V)}\rm
Let $V$ be an oriented vector space of dimension $8$. 
If $\bm{\Phi}\in\mathcal{A}(V)$, then we have the orthogonal decomposition
\begin{equation}\label{eq:ortho_decomp}
\wedge^4 V^*=T_{\bm{\Phi}}\mathcal{A}(V)\oplus T_{\bm{\Phi}}^\perp\mathcal{A}(V)
\end{equation}
with respect to the induced inner product $\mathbf{g}_{\bm{\Phi}}$.
We define a neighborhood 
$\mathcal{T}(V)$ of $\mathcal{A}(V)$ in $\wedge^4 V^*$ by
\begin{equation*}
\begin{aligned}
\mathcal{T}(V)=\Set{
\bm{\Phi}+\bm{\alpha}|
\bm{\Phi}\in\mathcal{A}(V)\text{ and }
\bm{\alpha}\in T_{\bm{\Phi}}^\perp\mathcal{A}(V)
\text{ with }\norm{\bm{\alpha}}_{\mathbf{g}_{\bm{\Phi}}}<\rho}.
\end{aligned}
\end{equation*}
We choose and fix a small constant $\rho$ so that 
any $\bm{\chi}\in\mathcal{T}(V)$ is uniquely written
as $\bm{\chi}=\bm{\Phi}+\bm{\alpha}$ with $\bm{\alpha}\in T_{\bm{\Phi}}^\perp\mathcal{A}(V)$.
Thus we can define the projection
\begin{equation*}
\Theta :\mathcal{T}(V)\longrightarrow\mathcal{A}(V),\qquad \bm{\chi}\longmapsto\bm{\Phi}.
\end{equation*}
\end{definition}
\begin{lemma}[Joyce \cite{Joyce00}, Proposition $10.5.4$]\label{lem:ASD-subspace}
Let $\bm{\Phi}\in\mathcal{A}(V)$ and 
$\wedge^4 V^*=\wedge^4_+ V^*\oplus \wedge^4_- V^*$ be the orthogonal decomposition with 
respect to $\mathbf{g}_{\bm{\Phi}}$, where $\wedge^4_+ V^*$ (resp. $\wedge^4_- V^*$) 
is the set of self-dual (resp. anti-self-dual) $4$-forms on $V$.
Then we have the following inclusion:
\begin{equation*}
\wedge^4_- V^*\subset T_{\bm{\Phi}}\mathcal{A}(V).
\end{equation*}
\end{lemma}
Now we define $\Spin$-structures on oriented $8$-manifolds.
\begin{definition}\rm
Let $M$ be an oriented $8$-manifold.
We define $\mathcal{A}(M)\longrightarrow M$ to be the fiber bundle whose fiber over $x$ is
$\mathcal{A}(T^*_x M)\subset\wedge^4 T^*_x M$. Then
$\Phi\in C^\infty (\wedge^4 T^* M)$ is a \emph{Cayley $4$-form}
or a \emph{$\Spin$-structure} on $M$ if
$\Phi\in C^\infty (\mathcal{A}(M))$, i.e.,
$\Phi$ is a smooth section of $\mathcal{A}(M)$.
If $\Phi$ is a $\Spin$-structure on $M$, then $\Phi$ induces a Riemannian metric $g_\Phi$
since $\restrict{\Phi}{x}$ for each $x\in M$ induces a positive-definite inner product
$g_{\restrict{\Phi}{x}}$ on $T_x M$. A $\Spin$-structure $\Phi$ on $M$ is said to be
\emph{torsion-free} if it is parallel with respect to the induced Riemannian metric $g_\Phi$,
i.e., $\nabla_{g_\Phi}\Phi =0$,
where $\nabla_{g_\Phi}$ is the Levi-Civita connection of $g_\Phi$.
\end{definition}
\begin{definition}\rm
Let $\Phi$ be a $\Spin$-structure on an oriented $8$-manifold $M$. We define $\mathcal{T}(M)$
be the the fiber bundle whose fiber over $x$ is $\mathcal{T}(T^*_x M)\subset\wedge^4 T^*_x M$. 
Then for the constant $\rho$ given in Definition $\ref{def:T(V)}$,
we have the well-defined projection $\Theta :\mathcal{A}(M)\longrightarrow\mathcal{T}(M)$.
Also, we see from Lemma $\ref{lem:ASD-subspace}$ that $\wedge^4_- T^*M\subset T_\Phi\mathcal{A}(M)$
as subbundles of $\wedge^4 T^*M$. 
\end{definition}
\begin{lemma}[Joyce, Proposition $10.5.9$]\label{lem:epsilons}
Let $\Phi$ be a $\Spin$-structure on $M$.
There exist such that $\epsilon_1,\epsilon_2,\epsilon_3$ independent of $M$ and $\Phi$,
such that the following is true.

If $\eta\in C^\infty (\wedge^4T^* M)$ satisfies $\Norm{\eta}_{C^0}\leqslant\epsilon_1$, then $\Phi +\eta\in \mathcal{T}(M)$.
For this $\eta$, $\Theta (\Phi +\eta )$ is well-defined and expanded as
\begin{equation}\label{eq:expansion}
\Theta (\Phi +\eta )=\Phi +p(\eta )-F(\eta ),
\end{equation}
where $p(\eta )$ is the linear term and $F(\eta )$ is the higher order term in $\eta$,
and for each $x\in M$, $\restrict{p(\eta )}{x}$ is the $T_\Phi\mathcal{A}(V)$-component of $\restrict{\eta}{x}$ 
in the orthogonal decomposition \eqref{eq:ortho_decomp} for $V=T^*_x M$. 
Also, we have the following pointwise estimates for any $\eta, \eta'\in C^\infty (\wedge^4 T^* M)$
with $\norm{\eta},\norm{\eta'}\leqslant\epsilon_1$:
\begin{equation*}
\begin{aligned}
\norm{F(\eta )-F(\eta')}\leqslant& \epsilon_2\norm{\eta -\eta'}(\norm{\eta}+\norm{\eta'}),\\
\norm{\nabla (F(\eta )-F(\eta'))}\leqslant &\epsilon_3 \{
\norm{\eta -\eta'}(\norm{\eta}+\norm{\eta'})\norm{\der\Phi}+\norm{\nabla (\eta -\eta')}(\norm{\eta}+\norm{\eta'})\\
&\phantom{\epsilon_3\{ }+\norm{\eta -\eta'}(\norm{\nabla\eta}+\norm{\nabla\eta'})\} .
\end{aligned}
\end{equation*}
Here all norms are measured by $g_\Phi$.
\end{lemma}

The following result is important in that it relates the holonomy contained in $\Spin$ 
with the $\der$-closedness of the $\Spin$-structure.
\begin{theorem}[Salamon \cite{Salamon89}, Lemma $12.4$]\label{thm:d-closed_Spin}
Let $M$ be an oriented $8$-manifold. Let $\Phi$ be a $\Spin$-structure on $M$ and $g_\Phi$
the induced Riemannian metric on $M$.
Then the following conditions are equivalent.
\renewcommand{\labelenumi}{\theenumi}
\renewcommand{\theenumi}{(\arabic{enumi})}
\begin{enumerate}
\item $\Phi$ is a torsion-free $\Spin$-structure, i.e., $\nabla_{g_\Phi}\Phi =0$.
\item $\der\Phi =0$.
\item The holonomy group $\Hol (g_\Phi )$ of $g_\Phi$ is contained in $\Spin$.
\end{enumerate}
\end{theorem}
Now suppose $\widetilde{\Phi}\in C^\infty(\mathcal{T}(M))$ with $\der\widetilde{\Phi}=0$.
We shall construct such a form $\widetilde{\Phi}$ in Section $\ref{sec:T-approx}$.
Then $\Phi =\Theta (\widetilde{\Phi})$ is a $\Spin$-structure on $M$.
If $\eta\in C^\infty(\wedge^4 T^*M)$ with $\Norm{\eta}_{C^0}\leqslant\epsilon_1$,
then $\Theta (\Phi +\eta )$ is expanded as in \eqref{eq:expansion}.
Setting $\phi =\widetilde{\Phi}-\Phi$ and using $\der\widetilde{\Phi}=0$, we have 
\begin{equation*}
\der\Theta (\Phi +\eta )=-\der\phi +\der p(\eta )-\der F(\eta ).
\end{equation*}
Thus the equation $\der\Theta (\Phi +\eta )=0$ for $\Theta (\Phi +\eta )$ to be a torsion-free 
$\Spin$-structure is equivalent to
\begin{equation}\label{eq:torsion-free_1}
\der p (\eta )=\der\phi +\der F(\eta ).
\end{equation}
In particular, we see from Lemma $\ref{lem:ASD-subspace}$ that 
if $\eta\in C^\infty (\wedge^4_-T^*M)$ then $p(\eta )=\eta$,
so that equation \eqref{eq:torsion-free_1} becomes
\begin{equation}\label{eq:torsion-free_2}
\der \eta=\der\phi +\der F(\eta ).
\end{equation}
Joyce proved by using the iteration method and 
$\der C^\infty (\wedge^4_- T^*M)=\der C^\infty (\wedge^4 T^*M)$
that equation \eqref{eq:torsion-free_2} has a solution $\eta\in C^\infty (\wedge^4_-T^*M)$
if $\phi$ is sufficiently small with respect to certain norms (see Theorem $\ref{thm:Spin_existence}$).
\begin{theorem}[Joyce \cite{Joyce00}, Theorem $10.6.1$]\label{thm:A-hat}
Let $(M,g)$ be a compact Riemannian $8$-manifold such that its holonomy group $\Hol(g)$ is contained in $\Spin$.
Then the $\Ahat$-genus $\Ahat (M)$ of $M$ satisfies
\begin{equation}\label{eq:A-hat}
48\Ahat (M)=3\tau (M)-\chi (M),
\end{equation}
where $\tau (M)$ and $\chi(M)$ is the signature and the Euler characteristic of $M$ respectively.
Moreover, if $M$ is simply-connected, then $\Ahat (M)$ is $1,2,3$ or $4$, and the holonomy group of $(M,g)$ is determined as
\begin{equation*}
\Hol (g)=\begin{cases}
\Spin&\text{if }\Ahat (M)=1,\\
\SU (4)&\text{if }\Ahat (M)=2,\\
\Symp (2)&\text{if }\Ahat (M)=3,\\
\Symp (1)\times\Symp (1)&\text{if }\Ahat (M)=4.
\end{cases}
\end{equation*}
\end{theorem}

\section{The Gluing Procedure}\label{sec:gluing_procedure}

\subsection{Compact complex manifolds with an anticanonical divisor}\label{sec:CMWAD}
We suppose that $\overline{X}$ is a compact complex manifold of dimension $m$,
and $D$ is a smooth irreducible anticanonical divisor on $\overline{X}$.
We recall some results in \cite{Doi09}, Sections $3.1$--$3.2$, and \cite{DY14}, Sections $3.1$--$3.2$.
\begin{lemma}\label{lem:coords_on_X}
Let $\overline{X}$ and $D$ be as above.
Then there exists a local coordinate system
$\{ U_\alpha ,(z_\alpha^1,\dots ,z_\alpha^{m-1},w_\alpha)\}$ on $\overline{X}$
such that
\renewcommand{\labelenumi}{\theenumi}
\renewcommand{\theenumi}{(\roman{enumi})}
\begin{enumerate}
\item $w_\alpha$ is a local defining function of $D$ on $U_\alpha$,
i.e., $D\cap U_\alpha =\{w_\alpha =0\}$, and
\item the $m$-forms $\displaystyle\Omega_\alpha =
\frac{\der w_\alpha}{w_\alpha}\wedge\der z_\alpha^1\wedge\dots\wedge\der z_\alpha^{m-1}$
on $U_\alpha\setminus D$ together yield a holomorphic volume form $\Omega$ on $X=\overline{X}\setminus D$.
\end{enumerate}
\end{lemma}
Next we shall see that $X=\overline{X}\setminus D$ is a cylindrical manifold whose structure is
induced from the holomorphic normal bundle $N=N_{D/\overline{X}}$ to $D$ in $\overline{X}$,
where the definition of cylindrical manifolds is given as follows.
\begin{definition}\label{def:cyl.mfd}\rm
Let $X$ be a noncompact differentiable manifold of dimension $r$.
Then $X$ is called a \emph{cylindrical manifold} or a \emph{manifold with a cylindrical end} 
if there exists a diffeomorphism 
$\pi:X\setminus X_0\longrightarrow\Sigma\times\R_+=\Set{(p,t)|p\in\Sigma,0<t<\infty}$ 
for some compact submanifold $X_0$ of dimension $n$ with boundary $\Sigma=\del X_0$. 
Also, extending $t$ smoothly to $X$ so that $t\leqslant 0$ on $X\setminus X_0$, 
we call $t$ a \emph{cylindrical parameter} on $X$.
\end{definition}
Let $(x_\alpha ,y_\alpha)$ be local coordinates on $V_\alpha =U_\alpha\cap D$,
such that $x_\alpha$ is the restriction of $z_\alpha$ to $V_\alpha$ and
$y_\alpha$ is a coordinate in the fiber direction.
Then one can see easily that $\der x_\alpha^1\wedge\dots\wedge\der x_\alpha^{m-1}$ on
$V_\alpha$ together yield a holomorphic volume form $\Omega_D$, which is also called the
\emph{Poincar\'{e} residue} of $\Omega$ along $D$.
Let $\Norm{\cdot}$ be the norm of a Hermitian bundle metric on $N$.
We can define a cylindrical parameter $t$ on $N$
by $t=-\frac{1}{2}\log\Norm{s}^2$ for $s\in N\setminus D$.
Then the local coordinates $(z_\alpha ,w_\alpha )$ on $X$ are asymptotic to
the local coordinates $(x_\alpha ,y_\alpha )$ on $N\setminus D$ in the following sense.

\begin{lemma}\label{lem:tub.nbd.thm}
There exists a diffeomorphism $\Phi$ from a neighborhood $V$ of the
zero section of $N$ containing $t^{-1}(\R_+ )$ to a tubular neighborhood $U$ of $D$
in $X$ such that $\Phi$ can be locally written as
\begin{equation*}
\begin{aligned}
z_\alpha &=x_\alpha + O(\norm{y_\alpha}^2)=x_\alpha +O(e^{-t}),\\
w_\alpha &=y_\alpha + O(\norm{y_\alpha}^2)=y_\alpha +O(e^{-t}),
\end{aligned}
\end{equation*}
where we multiply all $z_\alpha$ and $w_\alpha$ by a single constant
to ensure $t^{-1}(\R_+ )\subset V$ if necessary.
\end{lemma}
Hence $X$ is a cylindrical manifold with the cylindrical parameter $t$
via the diffeomorphism $\Phi$ given in the above lemma.
In particular, when $H^0(\overline{X},\mathcal{O}_{\overline{X}})=0$ and
$N_{D/\overline{X}}$ is trivial, we have a useful coordinate system near $D$.
\begin{lemma}[\cite{DY14}, Lemma $3.4$.]\label{lem:existence_w}
Let $(\overline{X},D)$ be as in Lemma $\ref{lem:coords_on_X}$. If
$H^1(\overline{X},\mathcal{O}_{\overline{X}})=0$ and the normal
bundle $N_{D/\overline{X}}$ is holomorphically trivial, then there
exists an open neighborhood $U_D$ of $D$ and a holomorphic function
$w$ on $U_D$ such that $w$ is a local defining function of $D$ on
$U_D$. Also, we may define the cylindrical parameter $t$ with
$t^{-1}(\R_+)\subset U_D$ by writing the fiber coordinate $y$ of
$N_{D/\overline{X}}$ as $y=\exp (-t-\I\theta )$.
\end{lemma}
\subsection{Admissible pairs and asymptotically cylindrical Ricci-flat K\"{a}hler manifolds}

\begin{definition}\rm
Let $X$ be a cylindrical manifold such that
$\pi :X\setminus X_0\longrightarrow\Sigma\times\R_+=\{(p,t)\}$
is a corresponding diffeomorphism.
If $g_\Sigma$ is a Riemannian metric on $\Sigma$, then it defines a cylindrical metric
$g_{\rm cyl}=g_\Sigma +\der t^2$ on $\Sigma\times\R_+$.
Then a complete Riemannian metric $g$ on $X$ is said to be \emph{asymptotically cylindrical}
(\emph{to} $(\Sigma\times\R_+,g_{\rm cyl})$) if $g$ satisfies 
for some cylindrical metric $g_{\rm cyl}=g_\Sigma +\der t^2$
\begin{equation*}
\norm{\nabla_{g_{\rm cyl}}^j(g-g_{\rm cyl})}_{g_{\rm cyl}}\longrightarrow 0
\qquad\text{as }t\longrightarrow\infty\qquad\text{for all }j\geqslant 0,
\end{equation*}
where we regarded $g_{\rm cyl}$ as a Riemannian metric on $X\setminus X_0$
via the diffeomorphism $\pi$.
Also, we call $(X,g)$ an \emph{asymptotically cylindrical manifold} and
$(\Sigma\times\R_+,g_{\rm cyl})$ the \emph{asymptotic model} of $(X,g)$.
\end{definition}
\begin{definition}\label{def:admissible}\rm
Let $\overline{X}$ be a complex orbifold with isolated singular points
$\Sing\overline{X}=\set{p_1,\dots ,p_k}$ and $D$ a divisor on $\overline{X}$.
Then $(\overline{X},D)$ is said to be an \emph{orbifold admissible pair} if the following conditions hold:
\begin{enumerate}
\item[(a)] $\overline{X}$ is a compact K\"ahler orbifold.
\item[(b)] $D$ is a smooth anticanonical divisor on $\overline{X}$ with $D\cap\Sing \overline{X}=\emptyset$.
\item[(c)] the normal bundle $N_{D/\overline{X}}$ is trivial.
\item[(d)] $\overline{X}$ and $\overline{X}\setminus (D\sqcup\Sing \overline{X})$ are simply-connected.
\item[(e)] Each $p\in\Sing \overline{X}$ has a neighborhood $U_p$ such that there exists a crepant resolution
$\widetilde{U}_p\dashrightarrow U_p$ at $p$.
\end{enumerate}
Throughout this paper, we shall consider the action of $\Z_4$ on $\C^4$ generated by
\begin{equation*}
(z_1,z_2,z_3,z_4)\longmapsto (\I z_1,\I z_2,\I z_3,\I z_4)\qquad\text{for}\quad (z_1,z_2,z_3,z_4)\in\C^4.
\end{equation*}
If each $U_p$ in condition (e) is isomorphic to $\C^4/\Z_4$, where the action of $\Z_4$ is given above, 
then we shall call $(\overline{X},D)$ 
an orbifold admissible pair with isolated singular points \emph{modelled on} $\C^4/\Z_4$.
This kind of orbifold admissible pair plays an important role later in constructing compact $\Spin$-manifolds.
\end{definition}
If $\overline{X}$ is smooth, then $\Sing\overline{X}=\emptyset$ and condition (e) is empty,
so that the above conditions reduce to the definition of admissible pairs which originates in 
Kovalev \cite{Kovalev03} and is also used in our papers \cite{DY14}, \cite{DY15}.
From the above conditions, we see that Lemmas \ref{lem:coords_on_X} and
\ref{lem:existence_w} apply to admissible pairs.
Also, from conditions (a) and (b), we see that $D$ is a compact K\"{a}hler manifold
with trivial canonical bundle.

\begin{theorem}[Tian-Yau \cite{TY90}, Kovalev \cite{Kovalev03}, Hein \cite{Hein10}]\label{thm:TYKH}
Let $(\overline{X},\omega' )$ be a compact K\"{a}hler manifold
and $m=\dim_\C\overline{X}$. 
If $(\overline{X},D)$ is an admissible pair, then the following is true.

It follows from Lemmas $\ref{lem:coords_on_X}$ and
$\ref{lem:existence_w}$, there exist a local coordinate system
$(U_{D,\alpha} ,(z_\alpha^1,\dots ,z_\alpha^{m-1},w))$ on a
neighborhood $U_D=\cup_\alpha U_{D,\alpha}$ of $D$ and a holomorphic
volume form $\Omega$ on $\overline{X}\setminus D$ such that
\begin{equation*}
\Omega =\frac{\der w}{w}\wedge\der z_\alpha^1\wedge\dots\wedge
\der z_\alpha^{m-1}\quad\text{on }U_{D,\alpha}\setminus D.
\end{equation*}
Let $\kappa_D$ be the unique Ricci-flat K\"{a}hler form on $D$ in the K\"{a}hler class
$[\restrict{\omega'}{D}]$.
Also let $(x_\alpha ,y)$ be local coordinates of $N_{D/\overline{X}}\setminus D$
as in Section $\ref{sec:CMWAD}$ and write $y$ as $y=\exp (-t-\I\theta )$.
Now define a holomorphic volume form $\Omega_{\rm cyl}$
and a cylindrical Ricci-flat K\"{a}hler form $\omega_{\rm cyl}$ on $N_{D/\overline{X}}\setminus D$ by
\begin{equation}\label{eq:TYKH_CYcyl}
\begin{aligned}
\Omega_{\rm cyl}&=\frac{\der y}{y}\wedge\der x_\alpha^1\wedge\dots\wedge\der x_\alpha^{m-1}
=(\der t +\I\der\theta)\wedge\Omega_D,\\
\omega_{\rm cyl}&=\kappa_D+\frac{\I}{2}\frac{\der y\wedge\der\overline{y}}{\norm{y}^2}
=\kappa_D+\der t\wedge\der\theta .
\end{aligned}
\end{equation}
Then there exist a holomorphic volume form $\Omega$ and an asymptotically cylindrical Ricci-flat K\"{a}hler form $\omega$ on
$X=\overline{X}\setminus D$ such that
\begin{equation*}
\begin{aligned}
\Omega -\Omega_{\rm cyl}=\der\zeta,\quad
&\omega -\omega_{\rm cyl}=\der\xi\quad\text{for some }\zeta\text{ and }\xi\text{ with}\\
\norm{\nabla_{g_{\rm cyl}}^j\zeta}_{g_{\rm cyl}}=O(e^{-\beta t}),\quad
&\norm{\nabla_{g_{\rm cyl}}^j\xi}_{g_{\rm cyl}}=O(e^{-\beta t})
\quad\text{for all }j\geqslant 0\text{ and }\beta\in(0,\min\set{1/2,\sqrt{\lambda_1}}),
\end{aligned}
\end{equation*}
where $\lambda_1$ is the first eigenvalue of the Laplacian $\Delta_{g_D+\der\theta^2}$
acting on $D\times S^1$ with $g_D$ the metric associated with $\kappa_D$.
\end{theorem}
A pair $(\Omega ,\omega )$ consisting of a holomorphic volume form $\Omega$ and a Ricci-flat K\"{a}hler form $\omega$
on an $m$-dimensional K\"{a}hler manifold
normalized so that
\begin{equation*}
\frac{\omega^m}{m!}=\frac{(\I)^{m^2}}{2^m}\Omega\wedge\overline{\Omega}\;(=\text{the volume form})
\end{equation*}
is called a \emph{Calabi-Yau structure}.
The above theorem states that there exists a Calabi-Yau structure $(\Omega ,\omega)$
on $X$ asymptotic to a cylindrical Calabi-Yau structure
$(\Omega_{\rm cyl},\omega_{\rm cyl})$ on $N_{D/\overline{X}}\setminus D$
if we multiply $\Omega$ by some constant.

\begin{theorem}\label{thm:TYKH_orb}
The statement in Theorem $\ref{thm:TYKH}$ also holds when $(\overline{X},D)$ is an orbifold admissible pair.
\end{theorem}
\begin{proof}
This is essentially the same as the modification of 
the Calabi-Yau theorem for compact orbifolds. See \cite{BG07}, Chapter $3.6$.
\end{proof}

\subsection{K\"{a}hler orbifolds with an antiholomorphic involution and $\Spin$ manifolds}
\subsubsection{Two basic examples of ALE $\Spin$-manifolds}\label{sec:R^8/G}
Let $\Phi_0$ be the standard $\Spin$-structure on $\R^8=\set{(x_1,x_2,\dots ,x_8)}$.
Let $\alpha ,\beta$ act on $\R^8$ by
\begin{equation*}
\begin{aligned}
\alpha :&(x_1,x_2,\dots ,x_8)\longmapsto (-x_2,x_1,-x_4,x_3,-x_6,x_5,-x_8,x_7),\\
\beta :&(x_1,x_2,\dots ,x_8)\longmapsto (x_3,-x_4,-x_1,x_2,x_7,-x_8,-x_5,x_6).
\end{aligned}
\end{equation*}
Then 
$\alpha ,\beta$ satisfy $\alpha^4=\beta^4=\id_{\R^8}, \alpha\beta =\beta\alpha^3$ and
$\alpha^*\Phi_0=\beta^*\Phi_0=\Phi_0$, so that the group $G=\braket{\alpha ,\beta}$ is a
subgroup of $\Spin$.
Define complex coordinates $(z_1,z_2,z_3,z_4)$ and $(w_1,w_2,w_3,w_4)$ on $\R^8$ by
\begin{equation*}
\begin{cases}
z_1=x_1+\I x_2\\
z_2=x_3+\I x_4\\
z_3=x_5+\I x_6\\
z_4=x_7+\I x_8,
\end{cases}
\qquad
\begin{cases}
w_1=-x_1+\I x_3\\
w_2=x_2+\I x_4\\
w_3=-x_5+\I x_7\\
w_4=x_6+\I x_8.
\end{cases}
\end{equation*}
Then the coordinates $(z_1,z_2,z_3,z_4)$ and $(w_1,w_2,w_3,w_4)$ define Calabi-Yau structures $(\omega_0,\Omega_0)$ 
and $(\omega'_0,\Omega'_0)$ on $\R^8$ by
\begin{equation*}
\begin{cases}
\omega_0=\frac{\I}{2}\sum_{i=1}^4\der z_i\wedge\der\overline{z}_i\\
\Omega_0=\der z_1\wedge\der z_2\wedge\der z_3\wedge\der z_4,
\end{cases}\qquad
\begin{cases}
\omega'_0=\frac{\I}{2}\sum_{i=1}^4\der w_i\wedge\der\overline{w}_i\\
\Omega'_0=\der w_1\wedge\der w_2\wedge\der w_3\wedge\der w_4,
\end{cases}
\end{equation*}
both of which induce the $\Spin$-structure $\Phi_0$ by
\begin{equation*}
\Phi_0=\frac{1}{2}\omega_0\wedge\omega_0+\Real\Omega_0=\frac{1}{2}\omega'_0\wedge\omega'_0+\Real\Omega'_0.
\end{equation*}
We see that $\alpha ,\beta$ act on these coodinates as
\begin{align*}
&\begin{cases}
\alpha :(z_1,z_2,z_3,z_4)\longmapsto (\I z_1,\I z_2,\I z_3,\I z_4)\\
\beta :(z_1,z_2,z_3,z_4)\longmapsto (\overline{z}_2,-\overline{z}_1,\overline{z}_4,-\overline{z}_3),
\end{cases}\\
&\begin{cases}
\alpha :(w_1,w_2,w_3,w_4)\longmapsto (\overline{w}_2,-\overline{w}_1,\overline{w}_4,-\overline{w}_3)\\
\beta :(w_1,w_2,w_3,w_4)\longmapsto (\I w_1,\I w_2,\I w_3,\I w_4).
\end{cases}
\end{align*}

Now we resolve the singularity of $\R^8/G$ in two ways.
Let us consider the action of $\alpha$ on $\C^4$ in the $z$-coordinates.
Then we have the following commutative diagram:
\begin{equation*}
\xymatrix{
\widetilde{\beta}\curvearrowright\quad \mathcal{Y}_1
\ar@<3.5ex>@{-->}[d]_{\text{crepant}}\ar@{>>}[r]
& \mathcal{X}_1\ar@{-->}^{\pi_1}[d] \\
\underline{\beta}\curvearrowright\C^4/\braket{\alpha}\ar@{>>}[r]&\R^8/G ,
}
\end{equation*}
where $\underline{\beta}$ is an antiholomorphic involution on $\C^4/\braket{\alpha}$ induced by $\beta$, 
and $\widetilde{\beta}$ is the lift of $\underline{\beta}$ which acts freely on $\mathcal{Y}_1$.
Since there exists an ALE Calabi-Yau structure $(\widetilde{\omega}_1,\widetilde{\Omega}_1)$ on $\mathcal{Y}_1$
with 
\begin{equation*}
\widetilde{\beta}^*\widetilde{\omega}_1=-\widetilde{\omega}_1,\qquad \widetilde{\beta}^*\widetilde{\Omega}_1=\overline{(\widetilde{\Omega}_1)},
\end{equation*}
the induced torsion-free $\Spin$-structure $\widetilde{\Phi}_1=\frac{1}{2}\widetilde{\omega}_1\wedge\widetilde{\omega}_1+\Real\widetilde{\Omega}_1$
pushes down to a torsion-free $\Spin$-structure $\Phi_1$ on $\mathcal{X}_1$. 
This gives a resolution of $\R^8/G$ by an ALE $\Spin$-manifold $(\mathcal{X}_1,\Phi_1)$.
Similarly, if we consider the action of $\beta$ on $\C^4$ in the $w$-coordinate, then we have 
\begin{equation*}
\xymatrix{
\widetilde{\alpha}\curvearrowright\quad \mathcal{Y}_2
\ar@<3.5ex>@{-->}[d]_{\text{crepant}}\ar@{>>}[r]
& \mathcal{X}_2\ar@{-->}^{\pi_2}[d] \\
\underline{\alpha}\curvearrowright\C^4/\braket{\beta}\ar@{>>}[r]&\R^8/G .
}
\end{equation*}
If we consider
\begin{align*}
\phi : (z_1,z_2,z_3,z_4)&\longmapsto (w_1,w_2,w_3,w_4),\qquad\text{that is},\\
(x_1,x_2,\dots ,x_8)&\longmapsto (-x_1,x_3,x_2,x_4,-x_5,x_7,x_6,x_8),
\end{align*}
then $\phi$ induces an isomorphism $\C^4/\braket{\alpha}\stackrel{\cong}{\longrightarrow}\C^4/\braket{\beta}$, 
which lifts to an isomorphism $\widetilde{\phi}:\mathcal{Y}_1\stackrel{\cong}{\longrightarrow}\mathcal{Y}_2$.
Let $\Phi_2$ be a $\Spin$-structure on $\mathcal{X}_2$ to which 
the $\Spin$-structure $({\widetilde{\phi}}^{-1})^*\widetilde{\Phi}_1$ on $\mathcal{Y}_2$ pushes down.
Then $(\mathcal{X}_2,\Phi_2)$ is also an ALE $\Spin$-manifold which resolves $\R^8/G$,
but $\mathcal{X}_1, \mathcal{X}_2$ are
\emph{topologically distinct} because $\phi$ does not commute with $\alpha ,\beta$, so that 
the isomorphism $\phi$ acts nontrivially on $\R^8/G$.
\begin{proposition}[Joyce \cite{Joyce00}, Section $15.1.1$]
Let $(\mathcal{X}_s,\Phi_s)$ for $s=1,2$ be ALE $\Spin$-manifolds as above.
Then the fundamental group of $\mathcal{X}_s$ is $\Z_2$, and 
\begin{equation}\label{eq:Betti_ALE}
b^i(\mathcal{X}_s)=\begin{cases}1&\text{if }i=0,4\\0&\text{otherwise},\end{cases}\qquad\text{so that}\quad
\chi (\mathcal{X}_s)=2.
\end{equation}
\end{proposition}
\subsubsection{Compatible antiholomorphic involutions on orbifold admissible pairs}
\begin{proposition}\label{prop:sigma_lift}
Let $X$ be a complex orbifold and $\sigma :X\longrightarrow X$ be an antiholomorphic involution.
Suppose $S$ is a complex submanifold of $X$ such that $\sigma$ preserves and acts freely on $S$.
Then $\sigma$ lifts to a unique antiholomorphic involution $\widetilde{\sigma}$ on 
the blow-up $\varpi :\Bl_S (X)\dashrightarrow X$ of $X$ along $S$ such that $\widetilde{\sigma}$ preserves and acts freely on
$\varpi^{-1}(S)$.
\end{proposition}

\begin{proof}
Let $m=\dim_\C X$ and $k=\dim_\C S$.
Fix a point $x\in S$.
It is enough to find a lift $\widetilde{\sigma}$ of $\sigma$ acting on a neighborhood of $\varpi^{-1}(x)$ in $\Bl_S(X)$.

First we consider local coordinates near $x$ and $\sigma (x)$ in $X$.
We can choose a neighborhood $U$ of $x\in S$ and local coordinates $(\mathbf{y},\mathbf{z})=(y_1,\dots ,y_k, z_1,\dots ,z_{m-k})$ on $U$
such that $S\cap U=\set{\mathbf{z}=\mathbf{0}}$.
We can similarly choose local coordinates $(\mathbf{y}',\mathbf{z}')=(y'_1,\dots ,y'_k,z'_1,\dots ,z'_{m-k})$ on $\sigma (U)$ such that 
$\sigma (S\cap U)=\set{\mathbf{z}'=\mathbf{0}}$ and 
\begin{equation*}
(\mathbf{y}',\mathbf{z}')=\sigma (\mathbf{y},\mathbf{z})=(\alpha (\mathbf{y},\mathbf{z}),\beta (\mathbf{y},\mathbf{z}))
\end{equation*}
for some antiholomorphic functions $\alpha :\C^m\longrightarrow\C^k$ and 
$\beta :\C^m\longrightarrow\C^{m-k}$.
Also, $\sigma (S)=S$ yields that for $(\mathbf{y},\mathbf{0})\in S\cap U$ 
\begin{equation}\label{sigma_rest_to_S}
\sigma (\mathbf{y},\mathbf{0})=(\alpha (\mathbf{y},\mathbf{0}),\mathbf{0}),\qquad
\text{that is,}\quad\beta (\mathbf{y},\mathbf{0})=\mathbf{0}.
\end{equation}
Since $\sigma$ is an antiholomorphic diffeomorphism on $X$, 
the matrix $\left(\Derbar_i\sigma_j(\mathbf{y},\mathbf{z})\right)_{1\leqslant i,j\leqslant m}$
is invertible for all $(\mathbf{y},\mathbf{z})\in U$,
where $\Derbar_i$ is the antiholomorphic partial differentiation with respect to the $i$-th variable.
In particular, it follows from \eqref{sigma_rest_to_S} that for $(\mathbf{y},\mathbf{0})\in S\cap U$ we have
\begin{equation*}
\left(\Derbar_i\sigma_j(\mathbf{y},\mathbf{0})\right)_{1\leqslant i,j\leqslant m}=
\begin{pmatrix}
\left(\Derbar_i\alpha_j(\mathbf{y},\mathbf{0})\right)_{1\leqslant i,j\leqslant k}&O\\
\left(\Derbar_{k+i}\alpha_j(\mathbf{y},\mathbf{0})\right)_{1\leqslant i\leqslant m-k,1\leqslant j\leqslant k}
&\left(\Derbar_{k+i}\beta_j(\mathbf{y},\mathbf{0})\right)_{1\leqslant i,j\leqslant m-k}
\end{pmatrix},
\end{equation*}
so that $\left(\Derbar_{k+i}\beta_j(\mathbf{y},\mathbf{0})\right)_{1\leqslant i,j\leqslant m-k}$ is also invertible.
Thus we can exapnd $\beta (\mathbf{y},\mathbf{z})$ for small $\mathbf{z}$ as
\begin{equation}\label{beta_expansion}
\beta (\mathbf{y},\mathbf{z})=\sum_{i=1}^{m-k}\Derbar_{k+i}\beta (\mathbf{y},\mathbf{0})\overline{z}_i+O(\norm{\mathbf{z}}^2).
\end{equation}

Next we consider local coordinates near $\varpi^{-1}(x)$ and $\varpi^{-1}(\sigma(x))$ in $\Bl_S(X)$.
Local coordinates of $\Bl_S(X)$ on $\varpi^{-1}(U)$ are written as
\begin{equation*}
\Set{(\mathbf{y},\mathbf{z},[\bm{\zeta}])
\in\C^m\times\C P^{m-k-1}|z_i\zeta_j =z_j\zeta_i
\text{ for all }i,j\in\set{1,\dots ,m-k}},
\end{equation*}
where $\bm{\zeta}=(\zeta_1,\dots ,\zeta_{m-k})\in\C^{m-k}$.
Similarly, local coordinates of $\Bl_S(X)$ on $\varpi^{-1}(\sigma (U))$ are written as
\begin{equation*}
\Set{(\mathbf{y}',\mathbf{z}',[\bm{\zeta}'])
\in\C^m\times\C P^{m-k-1}|z'_i\zeta'_j =z'_j\zeta'_i
\text{ for all }i,j\in\set{1,\dots ,m-k}}.
\end{equation*}
Thus we have
\begin{equation*}
\begin{array}{ll}
\varpi^{-1}(\mathbf{y},\mathbf{z})=\set{(\mathbf{y},\mathbf{z},[\mathbf{z}])}&
\text{for}\quad(\mathbf{y},\mathbf{z})\in U\setminus S\quad (\text{and so }\mathbf{z}\neq\mathbf{0}),\\
\varpi^{-1}(\mathbf{y},\mathbf{0})=\Set{(\mathbf{y},\mathbf{0},[\bm{\zeta}])|[\bm{\zeta}]\in\C P^{m-k-1}}&
\text{for}\quad(\mathbf{y},\mathbf{0})\in S\cap U.
\end{array}
\end{equation*}

Now we shall find a lift $\widetilde{\sigma}$ of $\sigma$ acting on $\varpi^{-1}(U)$.
For $(\mathbf{y},\mathbf{z})\in U\setminus S$, we must have
\begin{equation*}
\widetilde{\sigma}(\mathbf{y},\mathbf{z},[\mathbf{z}])=(\sigma (\mathbf{y},\mathbf{z}),[\beta (\mathbf{z})]).
\end{equation*}
Then $\widetilde{\sigma}$ extends naturally to $\varpi^{-1}(S\cap U)$ by continuity as
\begin{align*}
\widetilde{\sigma}(\mathbf{y},\mathbf{0},[\bm{\zeta}])&=\lim_{\lambda\to 0}\widetilde{\sigma}(\mathbf{y},\lambda\bm{\zeta},[\lambda\bm{\zeta}])\\
&=\lim_{\lambda\to 0}(\alpha (\mathbf{y},\lambda\bm{\zeta}),\beta (\mathbf{y},\lambda\bm{\zeta}),[\beta (\mathbf{y},\lambda\bm{\zeta})])\\
&=\left(\alpha (\mathbf{y},\mathbf{0}),\mathbf{0},
\left[\sum_{i=1}^{m-k}\Derbar_{k+i}\beta (\mathbf{y},\mathbf{0})\overline{\zeta}_i\right]\right),
\end{align*}
where we used the expansion in \eqref{beta_expansion} for the last equality.
This gives the desired action of $\widetilde{\sigma}$ on the neighborhood $\varpi^{-1}(U)$ of $\varpi^{-1}(x)$ in $\Bl_S(X)$.
\end{proof}
\begin{definition}\label{def:compatible}\rm
Let $\overline{X}$ be a four-dimensional compact K\"{a}hler orbifold with isolated singular points modelled on $\C^4/\Z_4$,
such that $(\overline{X},D)$ is an orbifold admissible pair.
An antiholomorphic involution $\sigma$ on $\overline{X}$ is said to be \emph{compatible with}
$(\overline{X},D)$ if the following conditions hold:
\begin{itemize}
\item[(f)] We can choose a defining function $w$ on a neighborhood $U_D$ of $D$ given in
Lemma $\ref{lem:existence_w}$ so that
\begin{equation}\label{eq:sigma^*_w}
\sigma^* w=\overline{w},
\end{equation}
where $\overline{f}$ for a complex function $f$ is defined by $\overline{f}(x)=\overline{f(x)}$.
\item[(g)] $(\overline{X})^{\sigma}=\Sing\overline{X}$, where $(\overline{X})^{\sigma}$ is the fixed point set of 
the action of $\sigma$ on $\overline{X}$.
\end{itemize}
\end{definition}
Note that \eqref{eq:sigma^*_w} in condition (f) implies $\sigma (D)=D$, and 
$\sigma_D=\restrict{\sigma}{D}$ yields an antiholomorphic involution on $D$.
\begin{lemma}\label{lem:sigma_cyl}
Let $\sigma_{\rm cyl}$ be an antiholomorphic involution on $N_{D/\overline{X}}$ defined by
\begin{equation}\label{eq:sigma_cyl}
\sigma_{\rm cyl} (x_\alpha ,y)=(\sigma_D(x_\alpha ),\overline{y})\qquad\text{for}\quad (x_\alpha ,y)
\in (U_\alpha \cap D)\times\C\subset N_{D/\overline{X}}.
\end{equation}
Then we have
\begin{equation*}
\sigma (z_\alpha ,w) =\sigma_{\rm cyl}(x_\alpha ,y)+O(e^{-t}).
\end{equation*}
\end{lemma}
\begin{proof}
Using \eqref{eq:sigma^*_w}, we can write $\sigma (z_\alpha ,w)$ as
\begin{equation}\label{eq:sigma(z,w)}
\sigma (z_\alpha ,w)=(\sigma_1(z_\alpha ,w),\overline{w})\qquad\text{with}\quad \sigma_1(x_\alpha ,0)=\sigma_D(x_\alpha ).
\end{equation}
Thus the assertion follows from Lemma $\ref{lem:tub.nbd.thm}$ and \eqref{eq:sigma(z,w)}.
\end{proof}
Since the cylindrical parameter $t$ is defined by $y=\exp(-t-\I\theta )$, we have
\begin{equation*}
\sigma_{\rm cyl}^* t=t,\qquad\sigma_{\rm cyl}^*\theta =-\theta 
\end{equation*}
and thus 
\begin{equation*}\label{eq:sigma_CS}
(N_{D/\overline{X}}\setminus D)/\braket{\sigma_{\rm cyl}}\simeq \left( (D\times S^1)/\braket{\sigma_{D\times S^1,\rm cyl}}\right) \times\R_+,
\end{equation*}
where $\sigma_{D\times S^1,\rm cyl}$ acts on $D\times S^1$ as
\begin{equation}\label{eq:sigma_cyl_D} 
\sigma_{D\times S^1,\rm cyl}(x_\alpha ,\theta )=(\sigma_D(x_\alpha ),-\theta ).
\end{equation}
One can prove the following result by Theorem $\ref{thm:TYKH_orb}$ and 
an argument as used in the proof of \cite{Joyce00}, Proposition $15.2.2$.
\begin{theorem}\label{thm:Spin_pushdown}
Let $(\overline{X},\omega')$ be a four-dimensional K\"{a}hler orbifold with isolated singular points modelled on $\C^4/\Z_4$, such that
$(\overline{X},D)$ is an orbifold admissible pair with a compatible antiholomorphic involution $\sigma$.
Then there exists an asymptotically cylindrical Calabi-Yau structure $(\omega ,\Omega )$ on $X=\overline{X}\setminus D$ 
asymptotic to $(\omega_{\rm cyl},\Omega_{\rm cyl})$ on $N\setminus D$, such that
\begin{equation*}
\sigma^*g=g,\qquad\sigma^*\omega =-\omega ,\qquad\sigma^*\Omega =\overline{\Omega},
\end{equation*}
where $N=N_{D/\overline{X}}$ and $g$ is the Riemannian metric on $X$ associated with $(\omega ,\Omega )$.
Thus the torsion-free $\Spin$-structure $\frac{1}{2}\omega\wedge\omega +\Real\Omega$ on $X$
pushes down to a torsion-free $\Spin$-structure $\Phi$ on $X/\braket{\sigma}$.
Also, an antiholomorphic involution $\sigma_{\rm cyl}$ defined in \eqref{eq:sigma_cyl}
satisfies
\begin{equation*}
\sigma^*_{\rm cyl}g_{\rm cyl}=g_{\rm cyl},\qquad\sigma_{\rm cyl}^*\omega_{\rm cyl} =-\omega_{\rm cyl} ,\qquad\sigma_{\rm cyl}^*\Omega_{\rm cyl} =
\overline{\Omega_{\rm cyl}},
\end{equation*}
so that the torsion-free $\Spin$-structure $\frac{1}{2}\omega_{\rm cyl}\wedge\omega_{\rm cyl}+\Real\Omega_{\rm cyl}$
pushes down to a torsion-free $\Spin$-structure $\Phi_{\rm cyl}$.
We have
\begin{equation}\label{eq:Phi-Phi_cyl}
\begin{aligned}
\Phi -\Phi_{\rm cyl}=\der\Xi ,&\qquad\text{for some }\Xi\text{ with}\\
\norm{\nabla_{g_{\rm cyl}}^j\Xi}_{g_{\rm cyl}}=O(e^{-\beta t}),&\qquad
\text{for all }j\geqslant 0\text{ and } \beta\in (0,\min\set{1/2,\sqrt{\lambda_1}}),
\end{aligned}
\end{equation}
where $\lambda_1$ is the constant given in Theorem $\ref{thm:TYKH}$.
Hence $(X/\braket{\sigma},\Phi)$ is an asymptotically cylindrical $\Spin$-manifold, 
with the asymptotic model 
\begin{equation*}
\begin{aligned}
(N\setminus D)/\braket{\sigma_{\rm cyl}}&\simeq  (D\times S^1)/\braket{\sigma_{D\times S^1,\rm cyl}} \times\R_+
=\set{([x_\alpha ,\theta],t)},\\
\text{where}\qquad [x_\alpha ,\theta ]&=[\sigma_D(x_\alpha ),-\theta ]\qquad\text{in}\quad (D\times S^1)/\braket{\sigma_{D\times S^1,\rm cyl}}.
\end{aligned}
\end{equation*}
\end{theorem}

\begin{theorem}[Joyce \cite{Joyce00}, Proposition $15.2.3$ and Corollary $15.2.4$]\label{thm:iota_p}
All isolated singular points in $X/\braket{\sigma}$ are modelled on $\R^8/G$ given in Section $\ref{sec:R^8/G}$. 
For each $p\in \Sing X/\braket{\sigma}$ there exists an isomorphism $\iota_p:\R^8/G\longrightarrow T_p(X/\braket{\sigma})$,
which identifies the $\Spin$-structures $\Phi_0$ on $\R^8$ and $\Phi$ on $T_p(X/\braket{\sigma})$.
\end{theorem}

\subsection{Gluing orbifold admissible pairs divided by compatible antiholomorphic involutions}\label{sec:gluing_admissible}
In this subsection we will only consider admissible pairs $(\overline{X},D)$ with
$\dim_\C\overline{X}=4$. Also, we will denote $N=N_{D/\overline{X}}$ and $X=\overline{X}\setminus D$.

\subsubsection{The gluing condition}\label{sec:gluing_cond}

Let $(\overline{X},\omega')$ be a four-dimensional compact K\"{a}hler orbifold with isolated singular points
modelled on $\C^4/\Z_4$,
and $(\overline{X},D)$ be an orbifold admissible pair with a compatible antiholomorphic involution $\sigma$.
Then we obtained in Theorem $\ref{thm:Spin_pushdown}$
an asymptotically cylindrical, torsion-free $\Spin$-manifold $(X,\Phi )$,
with the asymptotic model $(N\setminus D,\Phi_{\rm cyl})$.

Next we consider the condition under which we can glue together
$X_1/\braket{\sigma_1}$ and $X_2/\braket{\sigma_2}$ obtained from orbifold admissible pairs
$(\overline{X}_1,D_1)$ and $(\overline{X}_2,D_2)$ with antiholomorphic involutions $\sigma_i$. 
For gluing $X_1/\braket{\sigma_1}$ and $X_2/\braket{\sigma_2}$ to obtain a manifold with a $\Spin$-structure
with small torsion, we would like $(X_1/\braket{\sigma_1},\Phi_1 )$ and
$(X_2/\braket{\sigma_2},\Phi_2)$ to have the same asymptotic model. Thus we put the
following
\begin{description}
\item[\it Gluing condition] There exists an isomorphism
$\widetilde{f}: D_1\longrightarrow D_2$
between the cross-sections of the cylindrical ends of $\overline{X}_i\setminus D_i$ with 
\begin{equation*}
\widetilde{f}\circ\restrict{\sigma_1}{D_1}=\restrict{\sigma_2}{D_2}\circ\widetilde{f},
\end{equation*}
such that 
\begin{equation}\label{eq:gluing_condition}
\widetilde{f}_T^*\left(\frac{1}{2}\omega_{2,\rm cyl}\wedge\omega_{2,\rm cyl}+\Real\Omega_{2,\rm cyl}\right)
 =\frac{1}{2}\omega_{1,\rm cyl}\wedge\omega_{1,\rm cyl}+\Real\Omega_{1,\rm cyl},
\end{equation}
where $\widetilde{f}_T:D_1\times S^1\times (0,2T)\longrightarrow D_2\times S^1\times (0,2T)$
is defined by
\begin{equation*}
\widetilde{f}_T(x_1,\theta_1, t)=(\widetilde{f}( x_1),-\theta_1,2T-t)\qquad\text{for }
(x_1,\theta_1, t)\in D_1\times S^1\times (0,2T) .
\end{equation*}
\end{description}
\begin{lemma}\label{lem:gluing_condition}
If $\widetilde{f}:D_1\longrightarrow D_2$ is an isomorphism satisfyling 
$\widetilde{f}\circ\restrict{\sigma_1}{D_1}=\restrict{\sigma_2}{D_2}\circ\widetilde{f}$ 
and $\widetilde{f}^*\kappa_{D_2}=\kappa_{D_1}$.
Then the gluing condition \eqref{eq:gluing_condition} holds,
where we change the sign of $\Omega_{2,\rm cyl}$
(and also the sign of $\Omega_2$ correspondingly).
\end{lemma}
\begin{proof}
It follows by a straightforward calculation
using \eqref{eq:TYKH_CYcyl} and Lemma $\ref{lem:sigma_cyl}$. 
\end{proof}
The above $\widetilde{f}$ and $\widetilde{f}_T$ pushes down to maps 
\begin{equation*}
\begin{aligned}
f:D_1/\braket{\sigma_{D_1}}&\longrightarrow D_2/\braket{\sigma_{D_2}},\\
f_T: (D_1\times S^1)/\braket{\sigma_{D_1\times S^1,\rm cyl}}\times (0,2T)&\longrightarrow 
 (D_2\times S^1)/\braket{\sigma_{D_2\times S^1,\rm cyl}}\times (0,2T),\\
\text{with}\qquad f([x_1])=([\widetilde{f}(x_1 )]),&\qquad f_T([x_1 ,\theta_1],t)=([\widetilde{f}(x_1),-\theta_1 ],2T-t)
\end{aligned}
\end{equation*}
such that 
\begin{equation*}
f_T^*\Phi_{2,\rm cyl}=\Phi_{1,\rm cyl}.
\end{equation*}

\subsubsection{$\Spin$-structures with small torsion}\label{sec:T-approx}
Now we shall glue $X_1/\braket{\sigma_1}$ and $X_2/\braket{\sigma_2}$ under the gluing condition \eqref{eq:gluing_condition}.
Let $\rho :\R\longrightarrow [0,1]$ denote a cut-off function
\begin{equation*}
\rho (x)=
\begin{cases}
1&\text{if }x\leqslant 0,\\
0&\text{if }x\geqslant 1,
\end{cases}
\end{equation*}
and define $\rho_T :\R\longrightarrow [0,1]$ by
\begin{equation*}
\rho_T (x)=\rho (x-T+1)=
\begin{cases}
1&\text{if }x\leqslant T-1,\\
0&\text{if }x\geqslant T.
\end{cases}
\end{equation*}
Setting an approximating Calabi-Yau structure $(\Omega_{i,T}, \omega_{i,T})$ on $X_i$ by
\begin{equation*}
\Omega_{i,T}=
\begin{cases}
\Omega_i -\der (1-\rho_{T-1})\zeta_i &\text{on }\{ t_i\leqslant T-1\} ,\\
\Omega_{i,\rm cyl}+\der\rho_{T-1}\zeta_i &\text{on }\{ t_i\geqslant T-2\}
\end{cases}
\end{equation*}
and similarly
\begin{equation*}
\omega_{i,T}=
\begin{cases}
\omega_i -\der (1-\rho_{T-1})\xi_i &\text{on }\{ t_i\leqslant T-1\} ,\\
\omega_{i,\rm cyl}+\der\rho_{T-1}\xi_i &\text{on }\{ t_i\geqslant T-2\},
\end{cases}
\end{equation*}
we can define a $\der$-closed $4$-form
$\widetilde{\Phi}_{i,T}$ on each $X_i/\braket{\sigma_i}$ by
\begin{equation*}
\widetilde{\Phi}_{i,T}={\pi_i}_*\left(\frac{1}{2}\omega_{i,T}\wedge\omega_{i,T} +\Real\Omega_T \right) ,
\end{equation*}
where $\pi_i:X_i\longrightarrow X_i/\braket{\sigma_i}$ are projections.
We see that $\widetilde{\Phi}_{i,T}$ satisfies
\begin{equation*}
\widetilde{\Phi}_{i,T}=
\begin{cases}
\Phi_i&\text{on }\{ t_i<T-2\} ,\\
\Phi_{i,\rm cyl}&\text{on }\{ t_i>T-1\}
\end{cases}
\end{equation*}
and from \eqref{eq:Phi-Phi_cyl} that
\begin{equation}\label{eq:difference_phi}
\norm{\widetilde{\Phi}_{i,T} -\Phi_{i,\rm cyl}}_{g_{\Phi_{i,\rm cyl}}}=O(e^{-\beta T})
\qquad\text{for all }\beta\in (0,\min\set{1/2,\sqrt{\lambda_1}}).
\end{equation}

Let $X_{1,T}=\{ t_1<T+1\}\subset X_1$ and $X_{2,T}=\{ t_2<T+1\}\subset X_2$.
We glue $X_{1,T}/\braket{\sigma_1}$ and $X_{2,T}/\braket{\sigma_2}$
along $\left( (D_1\times S^1)/\braket{\sigma_{D_1\times S^1,\rm cyl}}\right)\times\{ T-1<t_1<T+1\}\subset X_{1,T}/\braket{\sigma_1}$
and $\left( (D_2\times S^1)/\braket{\sigma_{D_2\times S^1,\rm cyl}}\right)\times\{ T-1<t_2<T+1\}\subset X_{2,T}/\braket{\sigma_2}$
to construct a compact $8$-orbifold  using the gluing map $f_T$
(more precisely, $F_T=\varphi_2\circ f_T\circ\varphi_1^{-1}$,
where $\varphi_1$ and $\varphi_2$
are the diffeomorphisms given in Lemma \ref{lem:tub.nbd.thm}).
We denote this orbifold by $M_T^{\triangledown}$ (the upper index $\triangledown$ indicates singularities to be resolved).
Also, we can glue together $\widetilde{\Phi}_{1,T}$ and $\widetilde{\Phi}_{2,T}$ to obtain a 
$\der$-closed $4$-form $\widetilde{\Phi}_T$ on $M_T^\triangledown$ by Lemma $\ref{lem:gluing_condition}$.
There exists a positive constant $T_*$ such that $\widetilde{\Phi}_T\in C^\infty (\mathcal{T}(M_T^\triangledown))$
for any $T$ with $T>T_*$. 
This $\widetilde{\Phi}_T$ is what was discussed right after Theorem $\ref{thm:d-closed_Spin}$, from which
we can define a $\Spin$-structure $\Phi_T$ \emph{with small torsion} by
$\Phi_T =\Theta (\widetilde{\Phi}_T)$.
Letting $\phi_T = \widetilde{\Phi}_T -\Phi_T$,
we have $\der\phi_T +\der\Phi_T=0$.
\begin{proposition}\label{prop:estimates_T}
Let $T>T_*$.
Then there exist constants $A_{p,k,\beta}$
independent of $T$ such that for $\beta\in (0,\min\set{1/2,\sqrt{\lambda_1}})$ we have
\begin{equation*}
\Norm{\phi_T}_{L^p_k}\leqslant A_{p,k,\beta}\, e^{-\beta T},
\end{equation*}
where all norms are measured using $g_{\Phi_T}$.
\end{proposition}
\begin{proof}
These estimates follow in a straightforward way from Theorem \ref{thm:TYKH}
and \eqref{eq:difference_phi}
by an argument similar to those in \cite{Doi09}, Section 3.5.
\end{proof}
\subsubsection{Resolving $M_T^\triangledown$ by ALE $\Spin$-manifolds $\mathcal{X}_1$ and $\mathcal{X}_2$}
The material in this section is taken from \cite{Joyce00} Secion $15.2.2$. 
Let $p\in \Sing M_T^\triangledown$ and $\iota_p:\R^8/G\longrightarrow T_p M_T^\triangledown$ as in Theorem $\ref{thm:iota_p}$.
Let $\exp_p:T_p M_T^\triangledown\longrightarrow M_T^\triangledown$ be the exponential map.
Then $\psi_p=\exp_p\circ\iota_p$ maps each ball $B_{2\zeta}(\R^8/G)$ of $2\zeta$ in $\R^8/G$
to a neighborhood of $p\in M_T^\triangledown$.
Choose $\zeta >0$ small so that $U_p=\exp_p\circ\iota_p(B_{2\zeta}(\R^8/G))$ satisfy $U_p\cap U_{p'}=\emptyset$ 
and $U_p\cap\set{t_i>T-2}=\emptyset$ for any $p,p'\in M_T^\triangledown$ with $p\neq p'$ and for any $T>T_*$.
\begin{proposition}[Joyce \cite{Joyce00}, Proposition $15.2.6$]
There exist a smooth $3$-form $\sigma_p$ on $B_{2\zeta}(\R^8/G)$ for each $p\in\Sing M_T^\triangledown$ 
and a constant $C_1>0$ independent of $T>T_*$,
such that 
\begin{equation*}
\psi_p^*\Phi_T^\triangledown -\Phi_0=\der\sigma_p,\qquad
\norm{\nabla^\ell\sigma_p}\leqslant C_1r^{3-\ell}\qquad\text{for }\ell=0,1,2
\end{equation*}
on $B_{2\zeta}(\R^8/G)$.
Here $\norm{\cdot}$ and $\nabla$ is defined by the metric $g_0$ induced by $\Phi_0$, 
and $r$ is the radius function on $\R^8/G$.
\end{proposition}
Let $\pi_s:\mathcal{X}_s\longrightarrow\R^8/G$ be the projections given in Section $\ref{sec:R^8/G}$.
For each $\epsilon\in (0,1]$ and $s=1,2$ let $\mathcal{X}_s^\epsilon =\mathcal{X}_s$, 
define a $\Spin$-structure $\Phi_s^\epsilon =\epsilon^4\Phi_s$ 
and define $\pi_s^\epsilon :\mathcal{X}_s^\epsilon\longrightarrow\R^8/G$ by $\pi_s^\epsilon =\epsilon\pi_s$.
Then $(\mathcal{X}_s^\epsilon ,\Phi_s^\epsilon )$ is an ALE $\Spin$-manifold asymptotic to $\R^8/G$.
\begin{proposition}[Joyce \cite{Joyce00}, equation $(15.6)$]
There exist a constant $C_2>0$ independent of $T>T_*$, and a smooth $3$-form $\tau_s^\epsilon$ 
on $(\R^8 /G)\setminus B_{\epsilon\zeta}(\R^8/G)$ such that 
\begin{equation*}
(\pi_s^\epsilon )_*{\Phi_s^\epsilon}-\Phi_0 =\der\tau_s^\epsilon ,\qquad
\norm{\nabla^\ell\tau_s^\epsilon}\leqslant C_2\epsilon^8r^{-7-\ell}\qquad\text{for }\ell=0,1,2
\end{equation*}
on $(\R^8 /G)\setminus B_{\epsilon\zeta}(\R^8/G)$, where $\norm{\cdot}$ and $\nabla$ is defined using the metric $g_0$ 
induced by $\Phi_0$.
\end{proposition}
Now we glue together
\begin{equation*}
U_T^\epsilon =M_T^\triangledown\setminus\bigcup_{p\in\Sing M_T^\triangledown}\psi_p(\overline{B}_{\epsilon^{4/5}\zeta}(\R^8/G))\qquad
\text{and}\qquad V_p^\epsilon =(\pi_{s_p}^\epsilon )^{-1}(B_{2\epsilon^{4/5}\zeta}(\R^8/G)),\quad s_p\in\set{1,2},
\end{equation*}
along the regions diffeomorphic to
\begin{equation*}
B_{2\epsilon^{4/5}\zeta}(\R^8/G)\setminus \overline{B}_{\epsilon^{4/5}\zeta}(\R^8/G)\qquad\text{in }\R^8/G,
\end{equation*}
to obtain a compact $8$-manifold $M_T^\epsilon$. 
Let $\eta$ be a smooth cut-off function with $\eta (x)=0$ for $x\leqslant\zeta$ 
and $\eta (x)=1$ for $x\geqslant 2\zeta$.
Choosing $s_p\in\set{1,2}$ for each $p\in\Sing M_T^\triangledown$, we can also glue the $\Spin$-structures 
$\Phi_T^\triangledown$ on $M_T^\triangledown$ and $\Phi_{s_p}^\epsilon$ on $\mathcal{X}_{s_p}^\epsilon$
to obtain a closed $4$-form $\widetilde{\Phi}_T^\epsilon$ on $M_T^\epsilon$ by
\begin{equation*}
\widetilde{\Phi}_T^\epsilon =
\Phi_0+\der\left(\eta (\epsilon^{-4/5}r)\sigma_p\right)
+\der\left( (1-\eta (\epsilon^{-4/5}r))\tau_{s_p}^\epsilon\right)
\qquad\text{on}\quad U_T^\epsilon\cap V_p^\epsilon .
\end{equation*}
Now we set $\epsilon =\exp (-\gamma T)$ for some constant $\gamma >0$ to be determined later,
and define $M^\epsilon=M_T^\epsilon$, $\widetilde{\Phi}^\epsilon=\widetilde{\Phi}_T^\epsilon$ 
and $U^\epsilon=U_T^\epsilon$.
\begin{proposition}[Joyce \cite{Joyce00}, Proposition $15.2.9$]\label{prop:fund.group}
If $s_p=1$ for all $p\in\Sing M_T^\triangledown$, then the fundamental group of $M^\epsilon$ is $\Z_2$.
Otherwise, $M^\epsilon$ is simply-connected.
\end{proposition}
\begin{lemma}[Joyce \cite{Joyce00}, Lemma $15.2.11$]\label{lem:C3}
There exists a constant $C_3>0$ independent of $T>T_*$ such that $\widetilde{\Phi}_T$ satisfies
\begin{equation*}
\norm{\widetilde{\Phi}^\epsilon-\Phi_0}\leqslant C_3\epsilon^{8/5},\qquad
\norm{\nabla (\widetilde{\Phi}^\epsilon-\Phi_0)}\leqslant \epsilon^{4/5}
\end{equation*}
on $U^\epsilon\cap V_p^\epsilon$ for all $p\in\Sing M_T^\triangledown$, 
where $\norm{\cdot}$ and $\nabla$ is defined using the metric $g_0$ induced by $\Phi_0$.
\end{lemma}
Letting $\Phi^\epsilon =\Theta (\widetilde{\Phi}^\epsilon)$ and $\phi^\epsilon =\widetilde{\Phi}^\epsilon -\Phi^\epsilon$,
we have $\der\phi^\epsilon +\der\Phi^\epsilon =0$. 
\begin{theorem}\label{thm:estimates_epsilon}
There exist a family $(M^\epsilon ,\Phi^\epsilon )$ of smooth $8$-manifolds with a $\Spin$-structure with small torsion 
and resolutions $\pi^\epsilon :M^\epsilon\longrightarrow M^\triangledown$ for $\epsilon\in (0,1]$ such that we have
\begin{enumerate}
\item[(i)] $\Norm{\phi^\epsilon}_{L^2}\leqslant\lambda\epsilon^{24/5}$ 
and $\Norm{\der\phi^\epsilon}_{L^{10}}\leqslant\lambda\epsilon^{36/25}$,
\item[(ii)] the injectivity radius $\delta (g)$ satisfies $\delta (g)\geqslant\mu\epsilon$, and
\item[(iii)] the Riemann curvature $R(g)$ satisfies $\Norm{R(g)}_{C^0}\leqslant\nu\epsilon^{-2}$,
\end{enumerate}
where all norms are measured using the metric $g^\epsilon$ on $M^\epsilon$ induced by $\Phi^\epsilon$.
\end{theorem}
\begin{proof}
The proof is almost the same as that of \cite{Joyce00}, Proposition $15.2.13$
except for the contributions from the cylinder, which is diffeomorphic to 
$\Sigma\times (0,2T)$ with $\Sigma=(D\times S^1)/\braket{\sigma_{D\times S^1,\rm cyl}}$.
Joyce proved the following estimates using Lemma $\ref{lem:C3}$:
\begin{equation*}
\sum_{p\in\Sing M_T^\triangledown}\int_{U^\epsilon\cap V_p^\epsilon}\norm{\phi^\epsilon}^2\leqslant\lambda^2\epsilon^{48/5},
\qquad\sum_{p\in M_T^\triangledown}\int_{U^\epsilon\cap V_p^\epsilon}\norm{\der\phi^\epsilon}^2\leqslant\lambda^{10}\epsilon^{72/5},
\end{equation*}
Meanwhile, Proposition $\ref{prop:estimates_T}$ gives
\begin{equation*}
\int_{\Sigma\times (0,2T)}\norm{\phi_T}^2\leqslant 2{A_\beta}^2 e^{-2\beta T},\qquad
\int_{\Sigma\times (0,2T)}\norm{\der\phi_T}^{10}\leqslant 2{A_\beta}^{10} e^{-10\beta T},
\end{equation*}
where we take $\beta\in (0,\max\set{1/2,\sqrt{\lambda_1}})$ and $A_\beta =\max\set{A_{2,0,\beta},A_{10,1,\beta}}$.
Now if we choose $\gamma >0$ for $\epsilon =e^{-\gamma T}$ so that $\frac{24}{5}\gamma\leqslant\beta$, 
then we have $e^{-2\beta T}\leqslant\epsilon^{48/5}$ and $e^{-10\beta T}\leqslant\epsilon^{72/5}$.
Summing up the above contributions
and redefining $\lambda$ to be 
$\max\set{(\lambda^2+2{A_\beta}^2)^{1/2},(\lambda^{10}+2{A_\beta}^{10})^{1/10}}$,
we see that condition (i) holds. Conditions (ii) and (iii) are obvious.
\end{proof}
\subsection{Gluing theorems}\label{sec:gluing_theorems}
First we give a gluing and a doubling construction of Calabi-Yau fourfolds from \emph{orbifold} admissible pairs, which
are generalizations of Theorem $3.10$ and Corollary $3.11$ in \cite{DY15}.
\begin{theorem}\label{thm:main_CY}
Let $(\overline{X}_1,\omega'_1)$ and $(\overline{X}_2,\omega'_2)$ be
compact K\"{a}hler orbifolds with $\dim_\C\overline{X}_i=4$
such that $(\overline{X}_1,D_1)$ and $(\overline{X}_2,D_2)$ are orbifold admissible pairs.
Suppose there exists an isomorphism $f: D_1\longrightarrow D_2$ such that
$f^*\kappa_2=\kappa_1$,
where $\kappa_i$ is the unique Ricci-flat K\"{a}hler form on $D_i$ in the K\"{a}hler class
$[\restrict{\omega'_i}{D_i}]$.
Then we can glue toghether the crepant resolutions of $X_1$ and $X_2$ along their cylindrical ends to obtain
a compact simply-connected $8$-manifold $M$. The manifold $M$ admits a Riemannian metric with holonomy 
contained in $\Spin$. Moreover, if $\Ahat (M)=2$, then $M$ is a Calabi-Yau fourfold, i.e., 
$M$ admits a Ricci-flat K\"{a}hler metric with holonomy $\SU (4)$.
\end{theorem}
\begin{corollary}\label{cor:doubling_CY}
Let $(\overline{X},D)$ be an orbifold admissible pair with $\dim_\C\overline{X}=4$.
Then we can glue two copies of the crepant resolution of $X$ along their cylindrical ends to obtain a compact simply-connected 
$8$-manifold $M$. Then $M$ admits a Riemannian metric with holonomy contained in $\Spin$.
If $\Ahat (M)=2$, then the manifold $M$ is a Calabi-Yau fourfold.
\end{corollary}
Next we give a gluing and a doubling construction of compact $\Spin$-manifolds.
\begin{theorem}\label{thm:main}
Let $(\overline{X}_1,\omega'_1)$ and $(\overline{X}_2,\omega'_2)$ be four-dimensional compact K\"{a}hler orbifolds 
with singular points modelled on $\C^4/\Z_4$, 
such that $(\overline{X}_i,D_i)$ are orbifold admissible pairs 
with a compatible antiholomorphic involution $\sigma_i$.
Suppose there exists an isomorphism $f:D_1\longrightarrow D_2$ 
such that $f\circ\restrict{\sigma_1}{D_1}=\restrict{\sigma_2}{D_2}\circ f$ and $f^*\kappa_2=\kappa_1$, where $\kappa_i$ is the 
unique Ricci-flat K\"{a}hler form in the K\"{a}hler class $[\restrict{\omega'_i}{D_i}]$.
Then we can glue together $X_1/\braket{\sigma_1}$ and $X_2/\braket{\sigma_2}$ along their cylindrical ends to obtain a compact $8$-orbifold $M^\triangledown$.
There exists a compact simply-connected $8$-manifold $M$ which resolves $M^\triangledown$ at 
$(\#\Sing\overline{X}_1 +\#\Sing\overline{X}_2)$ isolated singular points
such that $M$ admits a Riemannian metric with holonomy contained in $\Spin$.
Furthermore if $\widehat{A}(M)=1$, then $M$ is a compact $\Spin$-manifold. 
\end{theorem}
\begin{corollary}\label{cor:doubling_Spin}
Let $(\overline{X},\omega')$ be a four-dimensional K\"{a}hler orbifold with isolated singular points modelled on $\C^4/\Z_4$,
such that $(\overline{X},D)$ be an orbifold admissible pair with a compatible antiholomorphic involution $\sigma$.
Then we can glue together two copies of $X/\braket{\sigma}=(\overline{X}\setminus D)/\braket{\sigma}$ to obtain a compact $8$-orbifold $M^\triangledown$.
There exists a comapct simply-connected $8$-manifold $M$ which resolves $M^\triangledown$ at $2(\#\Sing\overline{X})$ isolated singular points
such that $M$ admits a Riemannian metric with holonomy contained in $\Spin$.
Furthermore if $\widehat{A}(M)=1$, then $M$ is a compact $\Spin$-manifold.
\end{corollary}
\begin{proof}[Proof of Theorem $\ref{thm:main}$]
By Proposition $\ref{prop:fund.group}$, there exists a choice
$\set{s_p\in\set{1,2}|p\in\Sing M^\triangledown}$ of resolutions by $\mathcal{X}_{s_p}$
such that $M=M^\epsilon$ is simply-connected.
The assertion for $\Ahat(M)=1$ in Theorem $\ref{thm:main}$ follows directly from Theorem $\ref{thm:A-hat}$.
Thus it remains to prove the existence of a torsion-free $\Spin$-structure on $M^\epsilon$
for sufficiently small $\epsilon\in (0,1]$. 
This is a consequence of the following
\begin{theorem}[Joyce \cite{Joyce00}, Theorem $13.6.1$]\label{thm:Spin_existence}
Let $\lambda ,\mu ,\nu$ be positive constants. Then there exists a positive constant
$\epsilon_*$ such that whenever $0<\epsilon<\epsilon_*$, the following is true.\\
Let $M$ be a compact $8$-manifold and $\Phi$ a $\Spin$-structure on $M$. 
Suppose $\phi$ is a smooth $4$-form on $M$ with $\der\Phi +\der\phi=0$,
and
\begin{enumerate}
\item $\Norm{\phi}_{L^2}\leqslant\lambda\epsilon^{13/3}$ 
and $\Norm{\der\phi}_{L^{10}}\leqslant\lambda\epsilon^{7/5}$,
\item the injectivity radius $\delta (g)$ satisfies $\delta (g)\geqslant\mu\epsilon$, and
\item the Riemann curvature $R(g)$ satisfies $\Norm{R(g)}_{C^0}\leqslant\nu\epsilon^{-2}$.
\end{enumerate}
Let $\epsilon_1$ be as in Lemma $\ref{lem:epsilons}$.
Then there exists $\eta\in C^\infty (\wedge^4T^*_-M)$ with $\Norm{\eta}_{C^0}<\epsilon_1$
such that $\der\Theta (\Phi +\eta )=0$. Hence the manifold $M$ admits a torsion-free
$\Spin$-structure $\Theta (\Phi +\eta)$.
\end{theorem}
If we set $\phi =\phi^\epsilon$, then $M^\epsilon$ and $\phi^\epsilon$ satisfy conditions (i)--(iii) 
in Theorem $\ref{thm:estimates_epsilon}$. 
Thus we can apply Theorem \ref{thm:Spin_existence} to prove that $\Phi^\epsilon$
can be deformed into a torsion-free $\Spin$-structure for sufficiently small $\epsilon\in (0,1]$.
This completes the proof of Theorem $\ref{thm:main}$.
\end{proof}

\section{New examples of compact $\Spin$-manifolds}\label{sec:NewEx}
In this and the following section, we give three examples of compact simply-connected 
$8$-manifolds $M$ 
obtained by Corollary $\ref{cor:doubling_Spin}$ and Theorem \ref{thm:CI_Spin(7)}. 
We will see that the $\widehat A$-genus of $M$ is $1$ in each case. This shows that $M$ is a compact $\Spin$-manifold.

 We shall calculate the Euler characteristics, Betti numbers and the signatures of these compact $\Spin$-manifolds. 
 The Euler characteristics of the resulting compact $\Spin$-manifolds are $912$, $1296$ and $1680$. 
 Of these  compact $\Spin$-manifolds, 
 we see that at least one with $\chi(M)=1680$ is a new example.   
 
\subsection{Preliminaries}\label{sec:prelim}
In order to find orbifold admissible pairs with a compatible antiholomorphic involution in 
Definitions $\ref{def:admissible}$ and $\ref{def:compatible}$ we will use some algebro-geometrical approach. 
In particular, hypersurfaces and complete intersections in weighted projective spaces are well-studied in the context of mirror symmetry 
for Calabi-Yau manifolds. 
First we will review some basics on weighted projective spaces. 
Next we will also explain notation on complete intersections in weighted projective spaces. See \cite{Fletcher00} for more details.   
 
 \subsubsection{Basics on projective spaces.}\label{sec:proj sp}
First we will observe the structure of the weighted projective space as a complex orbifold. 
Let $a_0,\dots , a_n$ be positive integers with $\gcd (a_0, \dots, a_n)=1$. 
Recall that the \emph{weighted projective space} $\C P^n(a_0,\dots,a_n)$ is the quotient $(\C^{n+1}\setminus \{0\})/ \C^*$, 
where $\C^*$ acts on $\C^{n+1}\setminus \{0\}$ by
\begin{equation*}
\C^{n+1}\setminus \{0\} \longrightarrow \C^{n+1} \setminus \{0\} ,\qquad 
\quad (w_0, \dots, w_n) \longmapsto (t^{a_0}w_0, \dots, t^{a_n}w_n)
\end{equation*}
for $t\in \C^*$. Let us fix the point $p=[1,0,\dots,0]$ in $\C P^n(a_0,\dots,a_n)$.
Denote the stabilizer of $p$ in $\C^*$ by $(\C^*)_{p}$. 
Then the point $(1,0,\dots,0)$ in $\C^{n+1}\setminus\{ 0\}$ is taken to $(t^{a_0},0,\dots,0)$ 
under the action of $t\in \C^*$. 
Thus we have an isomorphism
\begin{equation*} 
(\C^*)_{p}=\Set{t\in \C^* | t^{a_{0}}=1}\cong \mathbb{Z}_{a_{0}},
\end{equation*}
where $\mathbb{Z}_{a_{0}}$ is a finite cyclic group of order $a_{0}$. Let $[z_0,\dots,z_n]$ be the weighted homogeneous coordinates on $\C P^n(a_0,\dots,a_n)$. 
Then the affine open chart
\begin{equation*}
U_{0}=\Set{[z_0,\dots,z_n]\in \C P^n(a_0,\dots,a_n) | z_{0}\neq 0}
\end{equation*}
is isomorphic to $\C^n/ \mathbb{Z}_{a_{0}}$ as follows. 
Taking an orbifold chart $\widetilde{U}_{0}$ on $\C P^n(a_0, \dots, a_n)$ 
with a continuous map $\varphi : \widetilde{U}_0 \longrightarrow U_0$, 
we consider affine coordinates $(x_1,\dots,x_n)$ on $\widetilde{U}_{0}$ with $x_j=z_j/z_{0}$. 
Then the stabilizer acts via 
\begin{equation}\label{map:quotient action}
 (x_1,\dots,x_n) \longmapsto (\z^{a_1}x_1,\dots,\z^{a_n}x_n),
\end{equation}
where $\z\in (\C^*)_{p}$ is a primitive $a_{0}$-th root of unity. This implies our desired result. 
Furthermore $p\in \C P^n(a_0,\dots,a_n)$ is a quotient singular point with a finite cyclic group 
$\mathbb{Z}_{a_{0}}$ which acts on $\C^n$ by \eqref{map:quotient action}. 
In particular, all singularities of $\C P^n(a_0,\dots,a_n)$ are cyclic quotient singularities.

Next we shall define $\C P^n(a_0,\dots,a_n)$ as a projective variety. 
Let $R$ be the graded ring $\C[z_0,\dots,z_n]$. Suppose each variable $z_i$ has the weight $a_i$. 
Then $R$ has a natural weight decomposition $\displaystyle R=\bigoplus _{d=0}^{\infty} R_d$ 
where $R_d$ denotes the vector space spanned by all monomials $z_0^{d_0}\dots z_n^{d_n}$ with $\sum a_id_i=d$. 
Elements of $R_d$ are said to be \emph{weighted homogeneous polynomials} of degree $d$ 
and then $\C P^n(a_0, \dots ,a_n)$ is defined by
\begin{equation*}
\C P^n(a_0,\dots ,a_n)=\Proj (R).
\end{equation*}
For a given finitely generated graded ring $R$, $\Proj (R)$ denotes the projective scheme 
which is an algebraic variety constructed by gluing affine varieties.  
Furthermore, if positive integers $a_1,\dots, a_n$ have a common divisor, we have an isomorphism 
\begin{equation*}
\C P^n(a_0,\dots ,a_n)\cong \C P^n(a_0, a_1/q, \dots, a_n/q) 
\end{equation*} 
where $q=\gcd(a_1,\dots,a_n)$. This yields the following property (see \cite{Fletcher00}, Corollary. 5.9):
\begin{proposition}\label{prop:wt proj} 
Let $a_0, \dots, a_n$ be positive integers with $\gcd (a_0,\dots,a_n)=1$. Then we have an isomorphism as varieties
\begin{equation*}
\C P^n(a_0,\dots ,a_n)\cong \C P^n(b_0, \dots, b_n) 
\end{equation*} 
for some positive integers $b_0, \dots, b_n$ with  $\gcd(b_0, \dots, \widehat{b_i},\dots, b_n)=1$ for each $i$.
Here the symbol $\widehat{b_i}$ means that the entry $b_i$ is omitted.
\end{proposition}
Hence it is natural to define the following.
\begin{definition}\label{def:well-formed1}\rm
A weighted projective space $\C P^n(a_0,\dots ,a_n)$ is said to be \emph{well-formed} if and only if
\linebreak$\gcd (a_0, \dots, \widehat{a_i}, \dots,a_n)=1$ for each $i$.
\end{definition}

Recall that the graded ring $R=\C[z_0,\dots, z_n]$ is given by $\deg z_i=a_i \in \mathbb{Z}_{> 0}$. 
Let $S=\C[w_0,\dots ,w_n]$ be the standard polynomial ring with $\deg w_i=1$. Then we have the injective ring homomorphism
\begin{equation*}
R \longrightarrow S,\qquad z_i \longmapsto w_i^{a_i}.  
\end{equation*}
This injective ring homomorphism induces the well-defined surjective morphism  of varieties
\begin{align}\label{map:quotient map1}
\pi: \quad \Proj (S)=\C P^n &\longrightarrow \Proj (R)=\C P^n(a_0,\dots,a_n) , \\
[w_0, \dots ,w_n] &\longmapsto [z_0,\dots,z_n]=[w_0^{a_0}, \dots , w_n^{a_n}] . \notag
\end{align}
By abuse of notation, we \emph{also} denote by $\pi$ the canonical projection from $\C^{n+1}\setminus \{0\}$ onto $\C P^n(a_0, \dots, a_n):$ 
\begin{equation*}\label{map:projection}
\pi :\quad \C^{n+1} \setminus \{0\} \longrightarrow \C P^n(a_0,\dots ,a_n), \qquad (w_0, \dots ,w_n) \longmapsto [w_0^{a_0}, \dots , w_n^{a_n}].
\end{equation*}
For this canonical projection $\pi$ and a subvariety $X\subset \C P^n(a_0,\dots,a_n)$, we define the \emph{affine cone} $C_X$ over $X$ to be  
 \begin{equation*}
 C_X={\pi}^{-1}(X)\cup \{0\} \quad in \quad \C^{n+1}.
 \end{equation*}
 \begin{definition}\label{def:quasismooth}\rm
 A subvariety $X$ of $\C P^n(a_0,\dots,a_n)$ is called \emph{quasismooth} if $C_X$ is smooth except at the origin.
 \end{definition}
 \begin{definition} \label{def:well-formed2}\rm
 Let $X$ be a subvariety of $\C P^n(a_0,\dots,a_n)$ with codimension $k$. 
 Then $X$ is said to be \emph{well-formed} if $\C P^n(a_0,\dots,a_n)$ is well-formed and 
$X$ does not contain a codimension $k+1$ singular locus of $\C P^n(a_0,\dots,a_n)$.
\end{definition}
 
\subsubsection{Weighted complete intersections}\label{sec:Wt CI}
\begin{definition}\label{def:wt complete int} {\rm{Let $a_0,\dots , a_n$ be positive integers with $\gcd (a_0, \dots, a_n)=1$ 
and $R=\C[z_0,\dots,z_n]$ be the graded ring with $\deg z_i=a_i$ as usual. Let $f_1,\dots,f_k$ with 
$k \leqslant n+1$ be 
weighted homogeneous polynomials of the graded ring $R$ with $\deg f_i=d_i$. 
Then $I=\braket{f_1,\dots,f_k}$ is a homogeneous ideal of $R$. We define $X_I$ by
\begin{equation*}
X_I=\Proj (R/I) \subset \C P^n(a_0,\dots,a_n).
\end{equation*}
Then $X_I$ is a \emph{weighted complete intersection} of multidegree $(d_1,\dots,d_k)$ 
if the defining ideal $I$ can be generated by a regular sequence $f_1,\dots,f_k$. 
Here a sequence of elements $f_1,\dots,f_k$ with $k \leqslant n+1$ in $R$ 
is said to be a \emph{regular sequence} 
if $f_1$ is not a zero-divisor in $R$ and the class $[f_i]$ is not a zero-divisor in 
$R/\braket{f_1,\dots,f_{i-1}}$ for each $2\leqslant i \leqslant k$.}}
\end{definition}
Now we will state the following results which will be needed for our argument later on.
\begin{lemma}[Fletcher \cite{Fletcher00}, Lemma. $7.1$]\label{lem:cohomology complete int}
Let $X\subset \C P^n(a_0, \dots, a_n)$ be a well-formed quasismooth weighted complete intersection 
with the defining ideal $I(X)=\braket{f_1,\dots,f_k}$. Suppose $\deg f_i=d_i$. Let $A$ be the residue ring
\begin{equation*}
A=\frac{\C [z_0,\dots,z_n]}{\braket{f_1,\dots,f_k}}.
\end{equation*}
Since each $f_i$ is homogeneous, the ring $A$ decomposes into each graded piece:$A=\bigoplus_m A_m$. Then
\begin{equation*}
H^q(X,\mathcal{O}_X(m)) \cong 
\begin{cases}
A_m & \text{if}\quad  q=0 \\
0     &  \text{if}\quad  q=1,\dots,\dim_{\C} X-1 \\
 A_{\a-m}      &  \text{if}\quad  q=\dim_{\C} X
\end{cases}
\end{equation*}
for all $m\in \mathbb{Z}$, where $\displaystyle \a=\sum_{\lambda=1}^k d_{\lambda}-\sum_{i=0}^{n} a_i$. 
\end{lemma}
In particular, we have the following beautiful result for hypersurfaces.
\begin{theorem}[Fletcher \cite{Fletcher00}, Theorem $7.2$]\label{th:cohomology hypersurf}
Let $f$ be the defining polynomial of a weighted hypersurface $X$ in 
$\C P^n(a_0,\dots,a_n)$ with $\deg f=d$.
The {\rm{Jacobian ring}} $R(f)$ of $f$ is the quotient ring
\begin{equation*}
R(f)=\frac{\C [z_0,\dots,z_n]}{\braket{\frac{\p f}{\p z_0},\dots,\frac{\p f}{\p z_n}}}.
\end{equation*}
Let $R(f)_m$ denote the $m$-th graded part of $R(f)$. Then the Hodge numbers of $X$ are given by
\begin{equation*}
h^{p,q}(X) = 
\begin{cases}
0 & \text{if}\quad  p+q\neq n-1, \;p\neq q \\
1     &  \text{if}\quad  p+q\neq n-1, \;p=q \\
 \dim_{\C} R(f)_{qd+\a}    &  \text{if}\quad  p+q=n-1, \;p\neq q \\
  \dim_{\C} R(f)_{qd+\a} +1   &  \text{if}\quad  p+q=n-1, \;p= q, \\
\end{cases}
\end{equation*}
where $\displaystyle \a=d-\sum_{i=0}^n a_i$.
\end{theorem}

\subsection{Orbifold admissible pairs with a compatible antiholomorphic involution
from weighted complete intersections}\label{sec:CI_WPS}
Here we consider a situation where the gluing condition holds naturally.
We first recall the following result, which provides a way of 
obtaining orbifold admissible pairs of \emph{Fano type}.
\begin{theorem}[Kovalev \cite{Kovalev03}]\label{thm:Fano}
Let $V$ be a Fano four-orbifold with isolated singular points which have local crepant resolutions, 
$D\in\norm{-K_V}$ a smooth Calabi-Yau divisor, and $S$ a
smooth surface in $D$ representing the self-intersection class of $D\cdot D$ on $V$.
Let $\varpi : \overline{X}\dasharrow V$ be the blow-up of $V$ along
the surface $S$. If we take the proper transform $D'$ of $D$ under the blow-up
$\varpi$, then $(\overline{X},D')$ is an orbifold admissible pair.
Moreover, $\restrict{\varpi}{D'}$ yields an isomorphism between $D'$ and $D$,
and so we may denote $D'$ by $D$. 
\end{theorem}
\begin{proof}
See \cite{Kovalev03}, Proposition $6.42$ and Corollary $6.43$. These results for Fano threefolds also hold for 
Fano four-orbifolds $V$.
\end{proof}
The above orbifold admissible pair $(\overline{X},D)$ obtained from $V$ and $D$ is said to be of Fano type.

Next we consider a well-formed weighted projective space $W:=\C P^{k+3}(a_0,a_1,\dots ,a_{k+3})$ 
with $k\geqslant 1$.
Let $f_1,\dots ,$ $f_{k+1}$ be a regular sequence of 
weighted homogeneous polynomials such that 
\begin{enumerate}
\item[(1)]
$\displaystyle\sum_{\lambda=1}^k d_\lambda=\sum_{i=0}^{k+3} a_i$, where $d_\lambda =\deg f_\lambda$, 
\item[(2)]
$V$ is a complete intersection defined by the ideal $I_{k-1}=\braket{f_1,\dots ,f_{k-1}}$, 
with isolated singular points modelled on $\C^4/\Z_4$ (we set $I_0=0$ and $V=W$ when $k=1$),
\item[(3)]
$D$ is a \emph{smooth} complete intersection defined by the ideal $I_k=\braket{f_1,\dots ,f_k}$, 
so that $D\cap\Sing V=\emptyset$, and
\item[(4)]
$S$ is a smooth complete intersection defined by the ideal $I_{k+1}=\braket{f_1,\dots ,f_{k+1}}$
with $\deg f_{k+1}=\deg f_k$.
\end{enumerate}
Then $V$ is a four-dimensional Fano orbifold with $D$ a smooth anticanonical Calabi-Yau divisor,
and $S$ is a smooth surface in $D$ representing $D\cdot D$ on $V$.
Suppose there exists an antiholomorphic involution $\sigma$ on $W$
such that
\begin{enumerate}
\item[(5)]
$\sigma^*f_i=\overline{f_i}$ for $i=1,\dots ,k+1$ and $\sigma$ acts freely on $D$ and $S$, and
\item[(6)]
$V^\sigma =\Sing V$, where $V^\sigma =\Set{x\in V|\sigma (x)=x}$.
\end{enumerate}
Then by Proposition $\ref{prop:sigma_lift}$, $\sigma$ lifts to
an antiholomorphic involution $\widetilde{\sigma}$ on the blow-up $\varpi :\overline{X}:=\Bl_S(V)\dashrightarrow V$
such that $\widetilde{\sigma}$ preserves and acts freely on the exceptional divisor $E:=\varpi^{-1}(S)$.
Let $[\mathbf{z}]=[z_0,\dots,z_4]$ be weighted homogeneous coordinates on $W$, 
with $\deg z_i=a_i$ for $i=0,\dots ,k+3$.
We can describe the blow-up $\overline{X}$ of $V$, the exceptional divisor $E$ 
and the proper transform $D'$ of $D$ as
\begin{align*}
\overline{X}=\Bl_{S}(V)&=\Set{([\mathbf{z}],[u,v])\in W\times\C P^1|
f_1(\mathbf{z})=\dots =f_{k-1}(\mathbf{z})=0, v f_k(\mathbf{z})=u f_{k+1}(\mathbf{z})},\\
\varpi&:\overline{X}\dashrightarrow V,\qquad ([\mathbf{z}],[u,v])\longmapsto [\mathbf{z}],\\
E=\varpi^{-1}(S)&=\Set{([\mathbf{z}],[u,v])\in W\times\C P^1|
f_1(\mathbf{z})=\dots =f_{k+1}(\mathbf{z})=0}\cong S\times\C P^1,\\
D'=\overline{\varpi^{-1}(D\setminus S)}
&=\Set{([\mathbf{z}],[u,v])\in W\times\C P^1|
f_1(\mathbf{z})=\dots =f_k(\mathbf{z})=u=0}\\
&=D\times\set{[0,1]\in\C P^1}\cong D,\\
E\cap D'&=S\times\set{[0,1]\in\C P^1}\cong S.
\end{align*}
Note that the above equation $v f_k(\mathbf{z})=u f_{k+1}(\mathbf{z})$ is well-defined because 
both $f_k(\mathbf{z})$ and $f_{k+1}(\mathbf{z})$ are sections of the line bundle $\mathcal{O}_W(d_k)$.
Also, we can compute as
\begin{align*}
D'&=\varpi^*D-E,\\
K_{\overline{X}}&=\varpi^* K_V+E=\varpi^* (K_V+D)-D'=-D',\\
N_{D'/\overline{X}}&=\restrict{D'}{D'}=D'\cdot D'=0.
\end{align*}
Let $\mathbf{z}'=(z'_0,\dots ,z'_{k+3})$ and consider the transformation
\begin{equation*}
z'_i=z_i\qquad\text{for}\quad i=0,1,\dots ,k+1,\qquad z'_{k+2}=f_k(\mathbf{z})\qquad\text{and}\qquad z'_{k+3}=f_{k+1}(\mathbf{z}).
\end{equation*}
Then $\mathbf{z}'$ define well-defined coordinates on $W$, and we can rewrite $\overline{X}$ and $D'$ as
\begin{align*}
\overline{X}&=\Set{([\mathbf{z}'],[u,v])\in W\times\C P^1|f'_1(\mathbf{z}')=\dots =f'_{k-1}(\mathbf{z}')=0,v z'_{k+2}=u z'_{k+3}},\\
D'&=\Set{([\mathbf{z}'],[u,v])\in\overline{X}|u=0},
\end{align*}
where $f'_i(\mathbf{z}')=f_i(\mathbf{z})$ for $i=1,\dots ,k-1$.
In this coordinate system, it follows from the proof of Proposition $\ref{prop:sigma_lift}$ that 
\begin{equation*}
\widetilde{\sigma}(\mathbf{z}',[u,v])=(\sigma (\mathbf{z}'),[\overline{u},\overline{v}])\qquad\text{for}\quad 
(\mathbf{z}',[u,v])\in\overline{X}.
\end{equation*}
Thus we may assume that the defining function $u$ of $D'$ on $\overline{X}$ satisfies \eqref{eq:sigma^*_w}, so that
$\widetilde{\sigma}$ is a compatible antiholomorphic involution on $\overline{X}$.

Now if $k\geqslant 2$ and we exchange $f_k$ and $f_{k-1}$ and correspondingly choose another $f_{k+1}$ in the above situation,
then $V$ and $\overline{X}$ may change, but $D$ and the asymptotic model of $\overline{X}\setminus D$ \emph{do not change}.
Let $(\overline{X}_1,D_1)$ and $\sigma_1$ be the former orbifold admissible pair 
$(\overline{X},D')$ and $\widetilde{\sigma}$, and $(\overline{X}_2,D_2),\sigma_2$ the latter.
Setting the isomorphism $\widetilde{f}:D_1\longrightarrow D_2$ by 
\begin{equation*}
\widetilde{f}=(\restrict{\varpi_2}{D_2})^{-1}\circ\id_{D}\circ\restrict{\varpi_1}{D_1}:D_1\longrightarrow D\longrightarrow D_2,
\end{equation*}
we have $\widetilde{f}\circ\restrict{\sigma_1}{D_1}=\restrict{\sigma_2}{D_2}\circ\widetilde{f}$.
Consequently we have the following
\begin{theorem}\label{thm:CI_Spin(7)}
The above isomorphism $\widetilde{f}$ satisfies the gluing condition given in Section $\ref{sec:gluing_cond}$.
Thus we can apply Theorem $\ref{thm:main}$ to $(\overline{X}_i,D_i),\sigma_i$ for $i=1,2$,
to obtain a compact simply-connected Riemanninan $8$-manifold $M$, 
which has holonomy $\Spin$ if $\widehat{A}(M)=1$.
\end{theorem}

\subsection{A simple example}\label{sec:Spin(7)}
In this subsection, we will find a simple example of compact simply-connected $8$-manifold $M$ 
admits a Riemannian metric with holonomy $\Spin$ constructed by Corollary $\ref{cor:doubling_Spin}$.
We will use the same notation as in Section $\ref{sec:CI_WPS}$.
 
\subsubsection{Setup}\label{sec:setup}
Let $W=\C P^4(1,1,1,1,4)$ be the weighted projective space and $[\mathbf{z}]=[z_0,\dots,z_4]$ be weighted homogeneous coordinates on $W$, 
with $\deg z_i=1$ for $0\leqslant i \leqslant 3$ and $\deg z_4=4$. 
Then $W$ has an isolated singular point $p=[0,0,0,0,1]$, which is modelled on $\C^4/\Z_4$.
If we define an holomorphic involution $\sigma$ on $W$ by
\begin{equation*}\label{map:antiholomorphic}
[z_0,z_1,z_2,z_3,z_4] \longmapsto [-\overline{z}_1,\overline{z}_0,-\overline{z}_3,\overline{z}_2,\overline{z}_4],
\end{equation*}
then we have $W^\sigma=\set{p}=\Sing W$. Define 
\begin{equation}\label{eq:wt complete int}
V=W,\qquad D=\Set{[\mathbf{z}]\in W|f_1(\mathbf{z})=0}\qquad\text{and}\qquad S=\Set{[\mathbf{z}]\in W|f_1(\mathbf{z})=f_2(\mathbf{z})=0}
\end{equation}
by weighted homogeneous polynomials
\begin{equation}\label{eq:wt hom poly}
f_1(\mathbf{z})=z_0^8+z_1^8+z_2^8+z_3^8+z_4^2
\qquad\text{and}\qquad
f_2(\mathbf{z})=az_0^8+az_1^8+bz_2^8+bz_3^8+cz_4^2,
\end{equation}
where $a,b$ and $c$ are real coefficients.
Then we see that conditions $(1)$--$(3)$, $(5)$ and $(6)$ in Section $\ref{sec:CI_WPS}$ hold.
Also, we can choose $a,b$ and $c$ so that condition $(4)$ holds.
Thus following Section $\ref{sec:CI_WPS}$, we have an orbifold admissible pair $(\overline{X},D)$
from $V,D$ and $S$, where $\overline{X}=\Bl_S(V)$ and we denote 
the proper transform $D'$ of $D$  by $D$ again.
Also, the lift of $\sigma$ on $\overline{X}$, which exists by Proposition $\ref{prop:sigma_lift}$ and is denoted by $\sigma$ again, 
satisfies conditions (f) and (g) in Definition $\ref{def:compatible}$,
so that $\sigma$ is a compatible antiholomorphic involution on $\overline{X}$.
Applying the doubling construction in Corollary \ref{cor:doubling_Spin}, 
we can resolve the orbifold $M^{\triangledown}=X/\braket{\sigma}\cup X/\braket{\sigma}$ to obtain a compact $8$-manifold $M$. 
Then we have the following result.
\begin{theorem}\label{th:new spin(7)}
This simply-connected $8$-manifold $M$ admits a Riemannian metric with holonomy $\Spin$. Moreover $M$ has
\begin{align*}
  \begin{cases}
  b^2(M)=b^3(M)=0, \\
  b^4(M)=1678, \\
  \chi(M)=1680 \qquad \text{and}\qquad \tau(M)=576.
  \end{cases}
   \end{align*}
  \end{theorem}
We will show this theorem in Section $\ref{sec:conclusion}$.

 \subsubsection{Contributions from the singular point}\label{sec:branch formula} 
 First, we observe that the branched covering of the isolated singular point $p=[0,0,0,0,1]$ in
 $V=\C P^4(1,1,1,1,4)$. Consider the surjective morphism
 \begin{equation*}
 \pi: \C P^4 \longrightarrow V
 \end{equation*}
 defined in \eqref{map:quotient map1}, and let $[\mathbf{w}]=[w_0,\dots ,w_4]$ be the standard homogeneous coordinates on $\C P^4$. 
 Then the restriction of the map $\pi$ to $\widetilde{\Sigma}_4:=\Set{[\mathbf{w}]\in \C P^4|w_4= 0}$ is bijective 
 since $\widetilde{\Sigma}_4$ can be identified with $\C P^3$. 
 On the other hand, the restriction of the map $\pi$ to
 $\widetilde{U}_p:=\set{[\mathbf{w}]\in \C P^4|w_4\neq 0}\cong \C^4$ is $4:1$ except at $p$. 
 This is because we have $U_p:=\set{[\mathbf{z}]\in V | z_4\neq 0}\cong \C^4/\mathbb{Z}_4$ 
 as seen in Section $\ref{sec:proj sp}$:
\begin{equation}\label{map:quotient map2}
\xymatrix{
\C P^4\ar@{>>}[d]^{\pi}&=&(\widetilde{\Sigma}_4\sqcup\set{p})\ar[d]_{1:1}&\sqcup&(\widetilde{U}_p\setminus \set{p})\ar[d]_{4:1}\\
V&=&(\Sigma_4\sqcup\set{p})&\sqcup&(U_p\setminus \set{p}).
}
\end{equation}
Here we denote $\Sigma_4=\pi(\widetilde{\Sigma}_4)=\Set{[\mathbf{z}]\in V |z_4=0}$.

Next we prove the following result.
 \begin{lemma}\label{lem:branch formula}
 Let $\widetilde{F}$ be a projective subvariety of $\C P^4$ with $\widetilde{F}\cap\set{p}=\emptyset$,
 and $F=\pi (\widetilde{F})$. Then we have
 \begin{equation*}
 \chi (F)=\frac{1}{4}(\chi(\widetilde{F})+3\chi(\widetilde{F}\cap\widetilde{\Sigma}_4)).
 \end{equation*}
 \end{lemma}
 \begin{proof}
The property of the map \eqref{map:quotient map2} yields
 \begin{align*}
 \chi(\widetilde{F})
 =& \chi(\widetilde{F} \setminus \widetilde{\Sigma}_4)+\chi(\widetilde{F}\cap\widetilde{\Sigma}_4) \\
 =& 4\chi(F\setminus\Sigma_4)+\chi(F\cap\Sigma_4) \\
 =& 4\chi(F)-3\chi(F\cap\Sigma_4)= 4\chi(F)-3\chi(\widetilde{F}\cap\widetilde{\Sigma}_4),
 \end{align*}
 where we used $\widetilde{F}\cap\widetilde{\Sigma}_4\cong F\cap\Sigma_4$ for the second and last equalities.
 Then arrangement shows our result.
 \end{proof}

\subsubsection{Computing the cohomology of $D$}\label{sec:topological invD}
In order to prove Theorem $\ref{th:new spin(7)}$ first we need to calculate the Euler characteristic $\chi(D)$ of the smooth Calabi-Yau divisor $D$. 
We will find this by following two ways.

\noindent\emph{Computing $\chi(D)$: First way.} 
Let $f_1$ and $f_2$ be the weighted homogeneous polynomial defined in \eqref{eq:wt hom poly}.
Then $\widetilde{f}_i=\pi^*f_i$ for $i=1,2$ are homogeneous polynomials of degree $8$ in $\C[w_0,\dots,w_4]$ given by
\begin{equation}\label{eq:wt hom poly2}
\widetilde{f}_1(\mathbf{w})=w_0^8+w_1^8+w_2^8+w_3^8+w_4^8,\qquad\text{and}\qquad\widetilde{f}_2(\mathbf{w})=aw_0^8+aw_1^8+bw_2^8+bw_3^8+cw_4^8,
\end{equation}
where $[\mathbf{w}]=[w_0,\dots ,w_4]$ 
are the standard homogeneous coordinates on $\C P^4$.
Setting
\begin{equation}\label{eq:D-S_tilde}
\widetilde{D}=\Set{[\mathbf{w}]\in \C P^4 | \widetilde{f}_1(\mathbf{w})=0}\qquad\text{and}\qquad
\widetilde{S}=\Set{[\mathbf{w}]\in \C P^4 | \widetilde{f}_1(\mathbf{w})=\widetilde{f}_2(\mathbf{w})=0},
\end{equation}
we have $\pi(\widetilde{D})=D$, $\pi(\widetilde{S})=S$ and $\widetilde{D}\cap\set{p}=\widetilde{S}\cap\set{p}=\emptyset$,
so that the assumption of Lemma $\ref{lem:branch formula}$ holds for $\widetilde{F}=\widetilde{D}, \widetilde{S}$.
Thus $\chi (D)$ is computed in terms of $\chi (\widetilde{D})$ and $\chi (\widetilde{D}\cap\widetilde{\Sigma}_4)$.
Since $\widetilde{D}\cap\widetilde{\Sigma}_4$ is given by
\begin{equation}\label{eq:DcapSigma}
\widetilde{D}\cap\widetilde{\Sigma}_4=\Set{[\mathbf{w}]\in\C P^4|\widetilde{f}_1(\mathbf{w})=w_4=0}
\cong\Set{[\mathbf{w}']\in\C P^3|w_0^8+w_1^8+w_2^8+w_3^8=0},
\end{equation}
where $[\mathbf{w}']=[w_0,w_1,w_2,w_3]$ are the standard homogeneous coordinates on $\C P^3$,
computing the total Chern classes gives
\begin{equation*}
\chi (\widetilde{D})=-2096\qquad\text{and}\qquad\chi (\widetilde{D}\cap\widetilde{\Sigma}_4)=7808,
\end{equation*}
which leads to the following result by Lemma $\ref{lem:branch formula}$.
\begin{proposition}\label{prop:topological invD}
This smooth Calabi-Yau divisor $D$ on $V$ has the Euler characteristic
\begin{equation*}
\chi(D)=-296.
\end{equation*}
\end{proposition}
\noindent\emph{Computing $\chi (D)$: Second way.} Theorem $\ref{th:cohomology hypersurf}$ determines the Hodge numbers of $D$ as follows. 
Let $R(f)$ be the Jacobian ring of $f$
\begin{equation*}
R(f)=\frac{\C[z_0,\dots,z_4]}{\braket{z_0^7,z_1^7,z_2^7,z_3^7,z_4}}.
\end{equation*}
Assume that a graded ring $B$ is finitely generated over $\C$. 
Then the \emph{Hilbert series} of the graded ring $B=\bigoplus_mB_m$ is defined to be
\begin{equation*}
 H_B(t)=\sum_{m=0}^{\infty}(\dim_{\C}B_m)t^m.
 \end{equation*}
 On the one hand, we can apply \cite{BT82}, Proposition $23.4$ to the Jacobian ring $R(f)$. 
 Consequently, the Hilbert series 
 of $R(f)$ is the power series expansion at $t=0$ of a rational function
 \begin{equation*}
 H_{R(f)}(t)=\frac{(1-t^7)^4}{(1-t)^4}=1+4t+10t^2+\dots +149t^8+\mathcal O(t^9).
  \end{equation*}
Then Theorem $\ref{th:cohomology hypersurf}$ gives 
 \begin{equation*}
 h^{3,0}(D)=\dim_{\C}R(f)_0=1\qquad \text{and}\qquad h^{2,1}(D)=\dim_{\C}R(f)_8=149 .
 \end{equation*}
 Thus the Hodge numbers of $D$ are
 \begin{equation*}
 h^{p,q}(D) =
\begin{array}{ccccccc}
&&& 1 &&& \\
&& 0 && 0 && \\
& 0 && 1 && 0 & \\
1 && 149 && 149 && 1 \\
& 0 && 1 && 0 & \\
&& 0 && 0 && \\
&&& 1 &&& \\
\end{array},
\hspace{2.5cm}
\begin{tabular}{c|c}
$i$ & $b^i$: Betti numbers \\

\hline

$0$ & $1$ \\
$1$ & $0$ \\
$2$ & $1$ \\
$3$ & $300$ \\
$4$ & $1$ \\
$5$ & $0$ \\
$6$ & $1$

\end{tabular}
\end{equation*}
Since the Euler characteristic $\chi(D)$ is also given by $\chi(D)=\sum_{p,q}(-1)^{p+q}h^{p,q}(D)$,
the result is consistent with Proposition \ref{prop:topological invD}.
\begin{remark}\rm
Theorem $\ref{th:cohomology hypersurf}$ is not essential to calculate the Hodge numbers in this example. 
In fact, we already know that $h^{0,0}=h^{3,0}=1$ since $D$ is a Calabi-Yau threefold 
(see \cite{Joyce00}, Proposition $6.2.6$). 
Therefore the Lefschetz hyperplane theorem and the Euler characteristic determine the Hodge numbers in this case.
 \end{remark}

 \subsubsection{Computing the cohomology of $S$}\label{sec:topological invS}
 Analogously to Section $\ref{sec:topological invD}$, we shall find all Hodge numbers of the 
 weighted complete intersection $S$ defined in \eqref{eq:wt complete int}. 
 
 Recall that $f_i(\mathbf{z})$ and $\widetilde{f}_i(\mathbf{w})$ for $i=1,2$
 are weighted homogeneous polynomials in $\C[z_0,\dots,z_4]$ 
 and $\C [w_0,\dots ,w_4]$
 given by \eqref{eq:wt hom poly} and \eqref{eq:wt hom poly2} respectively. 
 Also recall that the smooth complex surface $\widetilde{S}$ is a complete intersection given 
 in \eqref{eq:wt hom poly2} and \eqref{eq:D-S_tilde}, for which we have $\chi (\widetilde{S})=7808$.
 As in \eqref{eq:DcapSigma}, we have
 \begin{align*}
 \widetilde{S}\cap\widetilde{\Sigma}_4&=\Set{[\mathbf{w}]\in\C P^4|\widetilde{f}_1(\mathbf{w})=\widetilde{f}_2(\mathbf{w})=w_4=0}\\
 &\cong\Set{[\mathbf{w}']\in\C P^3|w_0^8+w_1^8+w_2^8+w_3^8=aw_0^8+aw_1^8+bw_2^8+bw_3^8=0},
 \end{align*}
 which is a smooth complex curve in $\widetilde{S}$ with $\chi(\widetilde{S}\cap\widetilde{\Sigma}_4)=-768$. 
Again by using Lemma $\ref{lem:branch formula}$, we find
 \begin{equation*}
 \chi(S)=1376.
 \end{equation*}
  Also, we have $b^1(S)=0$ by the Lefschetz hyperplane theorem. 
Let us consider the residue ring
\begin{equation*}
  A=\frac{\C[z_0,\dots,z_4]}{\braket{f_1,f_2}}.
  \end{equation*} 
 Using \cite{BT82}, Proposition $23.4$ again we find that the Hilbert series of $A$ can be written as 
  \begin{equation*}
  H_A(t)=\frac{(1-t^8)^2}{(1-t)^4(1-t^4)}=1+4t+10t^2+\dots+199t^8+\mathcal O(t^9).
  \end{equation*}
Applying Lemma $\ref{lem:cohomology complete int}$ to the residue ring $A$ for $q=2, m=0$ and $\alpha=8$, we have
  \begin{equation*}
  h^{0,2}(S)=\dim_{\C}A_8=199.
  \end{equation*}
Then the Hodge numbers of $S$ are 
   \begin{equation*}
 h^{p,q}(S) =
\begin{array}{ccccc}
&& 1 && \\
& 0 && 0 & \\
199 && 976 && 199 \\
& 0 && 0 & \\
&& 1 && \\
\end{array}\hspace{2.5cm}
\begin{tabular}{c|c}
$i$ & $b^i$: Betti numbers \\

\hline

$0$ & $1$ \\
$1$ & $0$ \\
$2$ & $1374$ \\
$3$ & $0$ \\
$4$ & $1$ \\
\end{tabular}
 \end{equation*}
since $\chi(S)=1376$. By the Hodge index theorem, we find the signature of $S$ is 
\begin{equation*}
\tau (S)=\sum_{p,q=0}^{\dim_{\C}S}(-1)^qh^{p,q}=-576.
\end{equation*}
Summing up our argument in Section $\ref{sec:topological invS}$, we have the following result.
\begin{proposition}\label{prop:topological invS}
This smooth compact complex surface $S$ has
\begin{equation*}
\chi(S)=1376 \qquad \text{and} \qquad \tau(S)=-576.
\end{equation*}
\end{proposition}
  
 \subsubsection{Conclusion of the above example.}\label{sec:conclusion}  
Now we are ready to prove Theorem $\ref{th:new spin(7)}$. We use the same notation as in Section $\ref{sec:setup}$.
The reader should refer to Section $4.4$ in \cite{DY15} as needed.
\begin{proof}[Proof of Theorem $\ref{th:new spin(7)}$]
First we will compute the Euler characteristic and the signature of the resulting compact simply-connected $8$-manifold $M$. 
Recall $\varpi : \overline{X} \dasharrow V$ is the blow-up of $V$ along the submanifold $S$. 
It is well-known that the Euler characteristic of $\overline{X}$ satisfies the equality
\begin{align}\label{eq:euler} 
\chi(\overline{X})=\chi(V)+\chi(E)-\chi(S)        
 \end{align}
 where $E$ is the exceptional divisor of the blow-up $\varpi$. 
As seen in Section $\ref{sec:CI_WPS}$, we have $E\cong S\times\C P^1$, and so
 \begin{align*}
 \chi(\overline{X})=\chi(V)+\chi(S)=1381,
 \end{align*} 
 where we used Proposition $\ref{prop:topological invS}$ and $\chi (V)=\chi (\C P^4(1,1,1,1,4))=5$. Thus 
 $\chi(X)=\chi(\overline{X})-\chi(D)=1677$. 
Since $\sigma$ fixes the singular point $p$ in $X$, we have 
\begin{equation*}
\chi (X/\braket{\sigma})=\frac{1}{2}(\chi (X)+1)=839.
\end{equation*}
Now we construct $M$ by resolving the orbifold $M^{\triangledown}=X/\braket{\sigma}\cup X/\braket{\sigma}$
with two isolated singlular points.
Observing from \eqref{eq:Betti_ALE} that replacing the neighborhood of each singular point in $M^\triangledown$ 
with an ALE manifold $\mathcal{X}_s$ adds $1$ to the Euler characteristic, we have
\begin{equation*}
\chi(M)=\chi(M^{\triangledown})+2=2\chi( X/\braket{\sigma})+2=1680.
\end{equation*}
To find the signature $\tau(M)$, we see that
\[
\tau(\overline{X})
=\tau(V)-\tau(S)=577
\]
in the same manner as \eqref{eq:euler}. Hence
\begin{align*}
\tau(M^{\triangledown})&=2\tau(X/\braket{\sigma})=\tau (X)+1\\
&=\frac{1}{2}(2\tau(\overline{X})-\tau (D\times\C P^1))+1=578.
\end{align*} 
Consequently we obtain our result
\begin{equation*}
\tau(M)=\tau(M^{\triangledown})-2=576
\end{equation*}
by taking resolutions of isolated singular points.
  
Next we shall show that $M$ admits a Riemannian metric with holonomy $\Spin$. 
This is a consequence of Theorem $\ref{thm:A-hat}$.
In fact, the Euler characteristic $\chi(M)=1680$ and the signature $\tau (M)=576$ above 
give the $\widehat A$-genus $\widehat{A}(M)=1$ for our example. 
Hence the assertion is verified.
  
 Finally we find the Betti numbers of our $\Spin$-manifold $M$. Consider
\[
M^{\triangledown}=Z_1 \cup Z_2
\]
where $Z_i=X/\braket{\sigma}$ for $i=1,2$. Then we have homotopy equivalences
\begin{equation}\label{eq:MV}
M^{\triangledown} \sim Z_1 \cup Z_2,\qquad Z_1 \cap Z_2 \sim (D \times S^1)/ \braket{\sigma_{D\times S^1,\rm cyl}} =: Y
\end{equation}
as in \cite{DY14}, equation $(4.6)$. Here the action of $\sigma_{D\times S^1,\rm cyl}$ is given by \eqref{eq:sigma_cyl_D}. 
\begin{lemma}[Kovalev \cite{Kovalev13}]\label{lem:Kov}
Let $Z_i \; (i=1,2)$ and $Y$ be as above. Then we have 
\begin{equation*}
b^1(Y)=b^2(Y)=0 \qquad \text{and} \qquad b^2(Z_i)=b^3(Z_i)=0.
\end{equation*}
\end{lemma}
Once Lemma $\ref{lem:Kov}$ has been proved, we conclude that
\[
b^2(M^{\triangledown})=b^3(M^{\triangledown})=0
\]
by applying the Mayer-Vietoris theorem to \eqref{eq:MV}. 
Then it follows from $\chi(M^{\triangledown})=1678$ that
\[
b^4(M^{\triangledown})=1676.
\]
By $(15.10)$ in \cite{Joyce00}, the Betti numbers $b^j(M)$ satisfy
\[
b^j(M)=b^j(M^{\triangledown}) \qquad \text{for}\quad  j=1,2,3 \qquad \text{and} \qquad b^4(M)=b^4(M^{\triangledown})+k
\]
where $k$ is the number of singlular points in $M^\triangledown$. 
Thus, we conclude our $\Spin$-manifold $M$ has the Betti numbers $(b^2,b^3,b^4)=(0,0,1678)$. 
In particular, this is a \emph{new} compact $\Spin$-manifold 
since no example with the same Betti numbers can be found among the known ones, 
all of which are listed in \cite{Joyce00} and \cite{Clancy11}.
 
In order to prove Lemma $\ref{lem:Kov}$, it suffices to show $b^2(Z_i)=b^3(Z_i)=0$ since $b^j(Y)=0$ for $j=1,2$ were already proved in \cite{Kovalev13}, 
 Proposition $6.2$. Recall that $b^2(V)=1$ and $b^3(V)=0$ because $V$ is $\C P^4(1,1,1,1,4)$. 
 Now $\varpi^{-1}(S)\cong S\times \C P^1$ 
 where $\varpi :\overline{X}\dasharrow V$ is the blow-up of $V$ along $S$. 
 Then the Betti numbers $b^i(\overline{X})$ are given by the formula
 \begin{equation*}\label{eq:blow-up formula}
 b^i(\overline{X})=b^i(V)+b^{i-2}(S)
 \end{equation*}
 (see \cite{DK87}, $(1.10)$). This gives
 \[
 b^2(\overline{X})=b^2(V)+b^0(S)=2 \qquad \text{and}\qquad b^3(\overline{X})=b^3(V)+b^1(S)=0.
 \]
 Since there is a tubular neighborhood $U$ of $D$ in $\overline{X}$ such that
 \begin{equation}\label{eq:MV2}
 \overline{X}=X\cup U \qquad \text{and}\qquad X\cap U \simeq D\times S^1 \times \R_{>0},
 \end{equation}
 we apply the Mayer-Vietoris theorem to \eqref{eq:MV2}. Then we see that
 \begin{equation}\label{eq:Betti}
 \begin{cases}
 b^2(\overline{X})=b^2(X)+1, \\
 b^3(X)=b^3(\overline{X})+b^2(D)-b^2(X)
 \end{cases}
 \end{equation}
 (see \cite{KL11}, $(2.10)$). Let $b^i(X)^{\sigma}$ be the dimension of the $\sigma$-invariant part of $H^i(X, \R)$. Then
 \[
 b^2(Z_i)=b^2(X)^{\sigma}=0
 \]
 because $H^2(X,\R)$ is generated by the K\"ahler form on $X$ and \emph{not} $\sigma$-invariant. Similarly,
 \[
 b^3(Z_i)=b^3(X)^{\sigma}=0
 \]
 by \eqref{eq:Betti} and the assertion is verified. 
  \end{proof}

\section{Other examples}\label{sec:other}
\setcounter{table}{2}

\subsection{Complete intersections in $\C P^5(1,1,1,1,4,4)$}\label{sec:CI_Spin(7)}

There are other examples based on the weighted complete intersection of two weighted hypersurfaces in $W= \C P^5(1,1,1,1,4,4)$ 
with homogeneous coordinates $[\mathbf{z}]=[z_0,\dots,z_5]$, where $\deg z_i=1$ for $0\leqslant i\leqslant 3$ and $\deg z_j=4$ for $j=4,5$. 
 Define an antiholomorphic involution $\sigma: W\longrightarrow W$ by
 \begin{equation*}\label{map:antiholomorphic2}
 [z_0,\dots,z_5]\longmapsto [-\overline{z}_1,\overline{z}_0,-\overline{z}_3,\overline{z}_2,\overline{z}_4,\overline{z}_5].
 \end{equation*}
Consider complete intersections 
\begin{align*}
V_1&=\Set{[\mathbf{z}]\in W | f_1(\mathbf{z})=0},\qquad D_1=\Set{[\mathbf{z}]\in W | f_1(\mathbf{z})=f_2(\mathbf{z})=0}\qquad\text{and}\\
S_1&=\Set{[\mathbf{z}]\in W | f_1(\mathbf{z})=f_2(\mathbf{z})=f_3(\mathbf{z})=0},
\end{align*}
where $f_1$ and $f_2$ are defined by
\begin{equation*}
f_1(\mathbf{z})=z_0^8+z_1^8+z_2^8+z_3^8+z_4^2-z_5^2
\qquad\text{and}\qquad
f_2(\mathbf{z})=z_0^4+z_1^4+z_2^4+z_3^4+2z_4+z_5,
\end{equation*}
and $f_3(\mathbf{z})$ is chosen so that $\deg f_3=\deg f_2=4$, $\sigma^*f_3=\overline{f_3}$, and 
$S_1$ is a smooth complete intersection in $W$.
Then $V_1$ has two isolated singular points $p_1=[0,0,0,0,1,1]$ and $p_2=[0,0,0,0,1,-1]$, which are modelled
on $\C^4/\Z_4$ and fixed by $\sigma$.
We can see easily that conditions $(1)$--$(6)$ in Section $\ref{sec:CI_WPS}$ hold, 
and thus following the argument in Section $\ref{sec:CI_WPS}$
we obtain an orbifold admissible pair $(\overline{X}_1,D_1)$ with a compatible antiholomorphic involution $\sigma_1$.

Similarly, we set $g_1=f_2$, $g_2=f_1$ and 
\begin{align*}
V_2&=\Set{[\mathbf{z}]\in W | g_1(\mathbf{z})=0},\qquad D_2=\Set{[\mathbf{z}]\in W | g_1(\mathbf{z})=g_2(\mathbf{z})=0}\qquad\text{and}\\
S_2&=\Set{[\mathbf{z}]\in W | g_1(\mathbf{z})=g_2(\mathbf{z})=g_3(\mathbf{z})=0},
\end{align*}
where we choose $g_3$ with $\deg g_3=\deg g_2=8$ so that $\sigma^*g_3=\overline{g_3}$,
and $S_2$ is a smooth complete intersection.
Then $V_2$ has an isolated singular point $p_3=[0,0,0,0,1,-2]$, which is modelled on $\C^4/\Z_4$ and fixed by $\sigma$.
Conditions $(1)$--$(6)$ in Section $\ref{sec:CI_WPS}$ also hold in this case, 
and we obtain another admissible pair $(\overline{X}_2,D_2)$ with $\sigma_2$.
Note that $(\overline{X}_i\setminus D_i)/\braket{\sigma_i}$ for $i=1,2$ have the same asymptotic model, and so can be glued together.

Now we can apply Corollary $\ref{cor:doubling_Spin}$ and Theorem \ref{thm:CI_Spin(7)}.
Setting $Z_i=(\overline{X}_i \setminus D_i)/\braket{\sigma_i}$ and $M_{ij}^{\triangledown}=Z_i \cup Z_j$, 
where $i,j\in\set{1,2}$, 
we can resolve orbifolds $M_{11}^{\triangledown} , M_{12}^{\triangledown}$ and
$M_{22}^{\triangledown}$ to obtain compact simply-connected $8$-manifolds 
 $M_{11}, M_{12}$ and $M_{22}$ respectively.
Then we see that $\widehat{A}(M_{ij})=1$ in each case. Hence we conclude that all resulting manifolds $M_{ij}$
are compact $\Spin$-manifolds. In particular,
the resulting manifold $M_{22}$ has the same Betti numbers as the above $\Spin$-manifold $M$ in Theorem $\ref{th:new spin(7)}$. 
Finally we shall list all Hodge numbers in Table $\ref{T1}$ which are needed to compute $\chi(M_{ij})$ and $\tau(M_{ij})$.
\begin{remark}\rm
Since our examples $M_{11},M_{12}$ with $(b^2,b^3,b^4)=(0,0,910), (0,0,1294)$ in Table $\ref{T2}$ are already listed in \cite{Joyce00}, 
we can not distinguish the topological types of these examples from those in \cite{Joyce00}.
\end{remark}

\subsection{From the viewpoint of Calabi-Yau structures}
In this subsection, 
we will give an example of Calabi-Yau fourfolds
constructed by Corollary $\ref{cor:doubling_CY}$. 
Ingredients of this example are exactly the same as the previous $\Spin$-manifold in Section $\ref{sec:Spin(7)}$, 
so that $V$ is $\C P^4(1,1,1,1,4)$, 
$D$ is an smooth anticanonical Calabi-Yau divisor on $V$, and $S$ is a smooth complex surface in $D$ representing $D\cdot D$.
Let $\varpi : \overline{X}\dasharrow V$ be the blow-up of $V$ along $S$. 
Let $D$ denote (again) the proper transform of $D\in\norm{-K_V}$ under the blow-up $\varpi$. 
Then we consider a `nice' resolution of the isolated singular point in $\overline{X}$ as follows. 
Recall that a complex algebraic variety $Y$ is said to be \emph{Gorenstein} if the canonical divisor $K_Y$ is a Cartier divisor. 
Suppose that $Y$ is Gorenstein and $\pi: \widehat{Y}\dasharrow Y$ is a resolution of $Y$. 
We say $\pi$ is a \emph{crepant resolution} if $\mathcal{O}(K_{\widehat{Y}})\cong \pi^{*}(\mathcal{O}(K_Y))$. 
Now the above $\overline{X}$ has an isolated singular point which is modelled on $\C^4/\mathbb{Z}_4$.
 Then toric geometry gives a method for finding all crepant resolutions. 
 See  \cite{Roan89}, and \cite{Joyce00}, Section $6.4$ for more details. 
 For example, the above $\C^4/\mathbb{Z}_4$ has a unique crepant resolution $\widehat{\C^4/\mathbb{Z}_4}$ which is isomorphic to $K_{\C P^3}$. 
 Let $\overline{\pi}: \widehat{X}\dasharrow \overline{X}$ 
 and $\pi :\widehat{V}\dasharrow V$ be
 the crepant resolutions of $\overline{X}$ and $V$ respectively which are given by the above argument. 
 Let $\widehat{D}$ denote the proper transform of $D\in\norm{-K_{\overline{X}}}$ under this resolution $\overline{\pi}$. 
 Then there is an induced map $\widehat{\varpi}: \widehat{X}\dasharrow \widehat{V}$ which makes the following diagram commutative:
\[
\xymatrix{
 \widehat{X}\ar@{-->}[d]_{\overline{\pi}: \text{ crepant}}\ar@{-->}[r]^-{\widehat{\varpi}}
& \widehat{V}\ar@{-->}[d]^{\pi : \text{ crepant}} \\
 \overline{X} \ar@{-->}[r]^-{\varpi} & V }
\]
Here the vertical maps are crepant resolutions and the horizontal maps are the blow-ups of four-dimensional complex algebraic varieties 
along the complete intersections. 
By gluing two copies of $\widehat{X}\setminus\widehat{D}$ 
along their cylindrical ends, we obtain a smooth compact $8$-manifold $M$. 
 \begin{proposition}
The above compact simply-connected smooth $8$-manifold $M$ is a Calabi-Yau fourfold.
 \end{proposition}
 \begin{proof}
 Again by Theorem $\ref{thm:A-hat}$, it suffices to see that $\widehat{A}({M})=2$. 
 Let $E$ be the exceptional divisor of the crepant resolution $\pi$, that is,
 \[
 \chi(E)=\chi(\widehat{\C^4/\mathbb{Z}_4})=\chi(K_{\C P^3})=4. 
 \]
Then we have the equalities of the Euler characteristics 
 \[
 \begin{cases}
 \chi(\widehat{X})=\chi(\overline{X})-1+\chi(E)=1381-1+4=1384, \\
 \chi(\widehat{D})=\chi(D)=-296.
 \end{cases}
 \]
 This implies 
 \begin{align*}
 \chi(M)=2(\chi(\widehat{X})-\chi(\widehat{D}))=3360.
 \end{align*}
 In order to find the signature $\tau(M)$, we see that $\tau(\widehat{X})=\tau(\overline{X})-1=576$. Hence we have
 \begin{align*}
 \tau({M})=2\tau(\widehat{X})-\tau(\widehat{D}\times \C P^1)=2\cdot 576-0=1152.
 \end{align*}
 Substituting $\chi(M)$ and $\tau(M)$ into \eqref{eq:A-hat} 
 we have $\widehat{A}(M)=2$.
 \end{proof}


\begin{center}
\begin{table}
\caption{The list of the Hodge diamonds}\label{T1}
\hspace{-2.6cm}{
 \newlength{\myheight}
\setlength{\myheight}{1.0cm}
 \newlength{\myheighta}
\setlength{\myheighta}{0.5cm}
 \begin{tabular}{c|ccccccccccccccccccccccccccccccc}  
\vspace{-0.25cm} &\multicolumn{9}{c}{Weighted hupersurfaces}  &&&&\multicolumn{7}{c}{Smooth Calabi-Yau} &&&& \multicolumn{7}{c}{\parbox[c][\myheighta][c]{0cm}{} \hspace{0.2cm}Weighted complete} \\
  {Index}\vspace{-0.25cm} &\multicolumn{9}{c}{}  &&&&\multicolumn{7}{c}{}  &&&& \multicolumn{7}{c}{} \\
    &\multicolumn{9}{c}{in $W$}  &&&&\multicolumn{7}{c}{divisors on $V_i$}  &&&& \multicolumn{7}{c}{\parbox[c][\myheighta][c]{0cm}{} \hspace{0.3cm}intersections in $V_i$} \\

 \hline
 
$i$    &\multicolumn{9}{c}{$V_i$}  &&&&\multicolumn{7}{c}{$D=D_1=D_2$} &&&&\multicolumn{7}{c}{\parbox[c][\myheight][c]{0cm}{}\hspace{0.2cm}
$S_i\in D_i\cdot D_i$} \\
 
 \hline
&  &&&&&&&&                               &&&&           &&&&&&                     &&&&   &&&&\\
&  &&&& 1 &&&&                           &&&&           &&&&&&                     &&&&    &&&& \\  
&  &&& 0 && 0 &&&                       &&&&          &&&1&&&                    &&&&       &&&& \\
&  && 0 && 1 && 0 &&                  &&&&        && 0 && 0 &&                 &&&&          &&1&& \\
&  & 0 && 0 && 0 && 0 &             &&&&      & 0 && 1 && 0 &               &&&&        &0&&0& \\ 
1&  0&&35 && 232 && 35 && 0   &&&&   1 && 149 && 149 && 1      &&&&     35&&232&&35 \\                 
&  & 0 && 0 && 0 && 0 &              &&&&      & 0 && 1 && 0 &               &&&&       &0&&0& \\                    
&  && 0 && 1 && 0 &&                  &&&&          && 0 && 0 &&                  &&&&         &&1&& \\
&  &&& 0 && 0 &&&                      &&&&            &&& 1 &&&                       &&&&          &&&&\\
&  &&&& 1 &&&&                             &&&&               &&&&&&                     &&&&             &&&&\\
&  &&&&&&&&                               &&&&           &&&&&&                     &&&&   &&&&\\
\hline
&  &&&&&&&&                               &&&&           &&&&&&                     &&&&   &&&&\\
&  &&&& 1 &&&&                           &&&&           &&&&&&                     &&&&    &&&& \\  
&  &&& 0 && 0 &&&                       &&&&          &&&1&&&                    &&&&       &&&& \\
&  && 0 && 1 && 0 &&                  &&&&        && 0 && 0 &&                 &&&&          &&1&& \\
&  & 0 && 0 && 0 && 0 &             &&&&      & 0 && 1 && 0 &               &&&&        &0&&0& \\ 
2&  0&&0 && 1 && 0 && 0   &&&&   1 && 149 && 149 && 1         &&&&     199&&976&&199 \\                 
&  & 0 && 0 && 0 && 0 &              &&&&      & 0 && 1 && 0 &               &&&&       &0&&0& \\                    
&  && 0 && 1 && 0 &&                  &&&&          && 0 && 0 &&                  &&&&         &&1&& \\
&  &&& 0 && 0 &&&                      &&&&            &&& 1 &&&                       &&&&          &&&&\\
&  &&&& 1 &&&&                             &&&&               &&&&&&                     &&&&             &&&&\\
&  &&&&&&&&                                 &&&&               &&&&&&                      &&&&            &&&&\\
\end{tabular}
  } 
 \end{table}
\end{center}

\begin{center}
\begin{table}[H]
\caption{The resulting $\Spin$-manifolds in Section $\ref{sec:CI_Spin(7)}$}\label{T2}
\setlength{\myheight}{1.0cm}
 \begin{tabular}{cccc}
         & & &\\
  \toprule
 The resulting   &   \vspace{-0.23cm}     &           &    \\
                           &   $\tau (M)$\vspace{-0.23cm}    &          $\chi(M)$       & $b^4$         \\
$\Spin$-manifolds $M$   &                             &        &      \\                           
 \midrule
 $M_{11}$                                                     &  $320$        &  $\phantom{1}912$  &  $\phantom{1}910$\\
 $M_{12}$                                                     &  $448$        &  $1296$  & $1294$\\
 $M_{22}$                                                     &  $576$        &  $1680$   & $1678$\\
 \bottomrule
 &&&
 \end{tabular}
\end{table}
  \end{center}

\end{document}